\documentclass[a4paper]{amsart}
\usepackage[english]{babel}

\usepackage{amssymb, amsmath,amsthm}
\usepackage{mathrsfs}
\usepackage{amscd}
\usepackage[active]{srcltx}
\usepackage{verbatim}
\usepackage{graphicx}
\usepackage{hyperref} 

\usepackage[utf8]{inputenc}
\usepackage[T1]{fontenc}
\usepackage{lmodern}

\usepackage{todonotes}
\usepackage{enumerate}
\usepackage{mathtools}
\usepackage{bbm}
\usepackage[normalem]{ulem}

\mathtoolsset{showonlyrefs,showmanualtags}

\newcommand\dd{\mathrm{d}}
\newcommand\ud{\,\mathrm{d}}

\newcommand\EE{\mathbb{E}}
\newcommand\Dd{\mathcal{D}}
\newcommand\Ss{\mathcal{S}}

\newcommand\Ff{\mathcal{F}}

\newcommand{\RR}{\mathbb{R}}

\newcommand{\NN}{\mathbb{N}}

\newcommand{\PP}{\mathbb{P}}

\newcommand{\inp}[2]{\langle #1,#2 \rangle}

\newcommand{\Nabla}{\nabla}
\newcommand{\Exp}{\mathrm{Exp}}

\newcommand{\Oo}{\mathcal{O}}

\newcommand{\cA}{\mathcal{A}}

\newcommand{\cE}{\mathcal{E}}

\newcommand{\cP}{\mathcal{P}}

\newcommand{\cX}{\mathcal{X}}

\renewcommand{\epsilon}{\varepsilon}

\theoremstyle{plain}
\newtheorem{theorem}{Theorem}[section]
\theoremstyle{definition}

\theoremstyle{remark}
\newtheorem{remark}[theorem]{Remark}
\theoremstyle{plain}
\newtheorem{corollary}[theorem]{Corollary}
\newtheorem{lemma}[theorem]{Lemma}
\newtheorem{proposition}[theorem]{Proposition}
\newtheorem{definition}[theorem]{Definition}

\numberwithin{equation}{section}

\allowdisplaybreaks

\begin{document}

\title{Large deviations for geodesic random walks}
\author{Rik Versendaal}
\address{Delft Institute of Applied Mathematics//Delft University of 
Technology\\P.O. Box 5031, 2600 GA Delft\\The Netherlands}
\email{R.Versendaal@TUDelft.nl}
\date\today

\begin{abstract}
We provide a direct proof of Cram\'er's theorem for geodesic random walks in a complete Riemannian manifold $(M,g)$. We show how to exploit the vector space structure of the tangent spaces to study large deviation properties of geodesic random walks in $M$. Furthermore, we reveal the geometric obstructions one runs into and overcome these by providing Taylor expansions of the inverse Riemannian exponential map, while also comparing the differential of the exponential map to parallel transport. Finally, we obtain the analogue of Cram\'er's theorem for geodesic random walks by showing that the curvature terms arising in this geometric analysis can be controlled and are negligible on an exponential scale.\\

\textit{keywords:} large deviations, Cram\'er's theorem, geodesic random walks, Riemannian exponential map, Jacobi fields
\end{abstract}

\thanks{The author is supported by the Peter Paul Peterich Foundation via TU Delft University Fund.}

\maketitle

\tableofcontents

\section{Introduction}\label{section:Introduction}

Random walks are among the most extensively studied discrete stochastic processes. Given a sequence of random variables $\{X_n\}_{n\geq1}$ in some vector space $V$, one defines the random walk with increments $\{X_n\}_{n\geq1}$ as the random variable
\[
S_n = \sum_{i=1}^n X_i.
\]
When rescaled by a factor $\frac1n$, one can study large deviations for the so obtained sequence $\{\frac1nS_n\}_{n\geq1}$. When the increments are independent and identically distributed, Cram\'er's theorem (\cite{DZ98,Hol00}) states that the sequence $\{\frac1nS_n\}_{n\geq 1}$ satisfies the large deviations principle. Intuitively, this means that there is some rate function $I:V \to [0,\infty]$ such that
\[
\PP\left(\frac1n\sum_{i=1}^n X_i \approx x\right) \approx e^{-nI(x)}.
\]
More specifically, the rate function is given as the Legendre transform of the log moment generating function of the increments, i.e.,
\[
I(x) = \sup_{\lambda} \left\{\inp{\lambda}{x} - \Lambda(\lambda)\right\},
\]
where $\Lambda(\lambda) = \log\EE(e^{\inp{\lambda}{X_1}})$. One may weaken the independence assumption to obtain for example the G\"artner-Ellis theorem, see e.g. \cite{DZ98,Hol00}. Also, Cram\'er's theorem can be generalized to the setting of topological vector spaces or Banach spaces. Furthermore, Cram\'er's theorem provides a basis for path space large deviations, such as Mogulskii's theorem (random walks) and Schilder's theorem (Brownian motion), see e.g. \cite{DZ98,Str84,DS89}. Recently, it was shown in \cite{KRV18} that the analogue of Cram\'er's theorem (as well as Mogulskii's theorem and Schilder's theorem) also holds in the Riemannian setting. \\

In \cite{KRV18}, Cram\'er's theorem for geodesic random walks is obtained by first proving the Riemannian analogue of Moguslkii's theorem, the path space analogue of Cram\'er's theorem. As evaluation in the end point of trajectories is a continuous map, Cram\'er's theorem then follows by an application of the contraction principle (see e.g. \cite[Chapter 4]{DZ98}). To obtain Mogulskii's theorem, the Feng-Kurtz formalism (\cite{FK06}) is used. However, this is the reverse order in which the theorems are obtained naturally in the Euclidean case. In the Euclidean setting, one uses Cram\'er's theorem to prove Mogulskii's theorem by first proving large deviations for the finite dimensional distributions and then deducing from these the large deviations on path space. Furthermore, the Feng-Kurtz approach is only suitable for Markov processes and hence does not extend to the case where the increments are allowed to be dependent. This causes an obstruction in finding a Riemannian analogue of the G\"artner-Ellis theorem for example.

These observations raise the question whether it is possible to avoid the use of the Feng-Kurtz formalism and path space large deviations to obtain Cram\'er's theorem for geodesic random walks. It turns out that it is possible to only study the underlying geometry in order to prove Cram\'er's theorem. This gives us new insight in what geometrical aspects allow us to still obtain the large deviation principle for rescaled geodesic random walks, even though the geodesic random walk is in general no longer a simple function of its increments. Furthermore, this approach does not rely on the fact that the random walk is a Markov process, and thus seems suitable to be extended to random walks with dependent increments for example. This will be investigated further in future work. \\

The main difficulty in the Riemannian setting, is that we lack a vector space structure to define a random walk as sum of increments. The appropriate analogue is a geodesic random walk as introduced by J\o rgensen in \cite{Jor75}. To define a geodesic random walk, we need to find a replacement for the additive structure, as well as a generalization of the increments. It turns out that as increments one uses tangent vectors, while the additive structure is replaced by an application of the Riemannian exponential map.

More precisely, we introduce a family of probability measures $\{\mu_x\}_{x\in M}$ such that for each $x \in M$, $\mu_x$ is a measure on $T_xM$, the tangent space at $x$. These measures $\{\mu_x\}_{x\in M}$ provide the space-dependent distribution of the increments. Now we start a random walk at some initial point $Z_0 = x_0 \in M$. Then recursively, we define for $k = 0,\ldots,n-1$ the random variable
\[
Z_{k+1} = \Exp_{Z_k}\left(\frac1nX_{k+1}\right),
\]
where $X_{k+1}$ is distributed according to $\mu_{Z_k}$. Hence, the random variable $Z_n$ takes values in $M$ and is the natural analogue of the empirical mean of the increments $X_1,\ldots,X_n$. In Euclidean space, this definition reduces to the usual one, as the Riemannian exponential map is simply vector addition, i.e.,
\[
\Exp_xv = x + v.
\]

To obtain an analogue of Cram\'er's theorem, we also need to generalize the notion of the increments of the random walk being identically distributed, since the increments are no longer in the same space. To compare two distributions $\mu_x$ and $\mu_y$, we need to identify the tangent spaces $T_xM$ and $T_yM$. We do this by taking a curve $\gamma$ connecting $x$ and $y$ and using parallel transport along $\gamma$. Because different curves lead to different identifications, we say that the distributions $\mu_x$ and $\mu_y$ are identical if for all curves $\gamma$ from $x$ to $y$ we have
\[
\mu_x = \mu_y \circ \tau^{-1}_{yx;\gamma},
\]
where $\tau$ denotes parallel transport. Equivalently, one can characterize this property by assuming that the log moment generating functions are invariant under parallel transport, i.e.,
\[
\Lambda_x(\lambda) = \Lambda_y(\tau_{xy;\gamma}\lambda),
\]
where $\Lambda_x(\lambda) = \log\int_{T_xM}e^{\inp{\lambda}{v}}\mu_x(\dd v)$.\\

In Euclidean space, the end point of the random walk is a simple function of the increments. In the Riemannian setting, curvature ensures that this is in general no longer the case. For example, the endpoint in general depends on the order of the increments. Nonetheless, it is possible to utilize the vector space structure of the tangent spaces. By controlling the error induced by the curvature, the large deviations for the geodesic random walk $Z_n$ can be obtained from the large deviations for $\frac1n\sum_{i=1}^n \tilde X_i$, the empirical mean of the appropriately pulled back increments in $T_{x_0}M$, were $x_0$ is the starting point of the random walk.

To support this claim, we can also define an alternative random walk in $M$. For this, we take a sequence of independent, identically distributed random variables $\{Y_n\}_{n\geq1}$ in $T_{x_0}M$ with distribution $\mu_{x_0}$ and consider the process 
\[
\tilde Z_n = \Exp_{x_0}\left(\frac1n\sum_{i=1}^n Y_i\right).
\]
In general, $\tilde Z_n$ is different from $Z_n$, even in distribution. Although our method of proving the large deviations for $Z_n$ does not immediately allow us to conclude that $Z_n$ and $\tilde Z_n$ are exponentially equivalent, the main idea of our proof does rely on the fact that we can (in some sense) relate and compare the geodesic random walk to a sum of independent, identically distributed random variables in the tangent space at $x_0$, following the distribution $\mu_{x_0}$.\\

The paper is organised as follows. In Section \ref{section:notation} we introduce the main notions we use from large deviation theory to obtain our results, as well as some notation and results from differential geometry. Section \ref{section:geodesic_random_walks} introduces the geodesic random walks. In Section \ref{section:approach_proof} we give the precise statement of Cram\'er's theorem for geodesic random walks. Additionally, we provide an overview of the various steps that are needed for the proof. In Section \ref{section:geometry} we obtain a Taylor expansions of the Riemannian exponential map with appropriate error bound. Furthermore, we compare the differential of the exponential map to parallel transport. Finally, we also provide bounds for how far geodesics, possibly starting at different points, can spread in a given amount of time. These geometric results are key ingredients in the proof of Cram\'er's theorem, which is given in Section \ref{section:Cramer_bounded}.


\section{Notation and important notions}\label{section:notation}

In this section we collect some important notions and fix the notation we will be using. Firstly, we introduce large deviation principles, along with some general useful results from the theory. Following up, we introduce the necessary tools from differential geometry and fix the notation for the various objects.

\subsection{Large deviation principle}

Large deviation principles are concerned with the asymptotic behaviour on an exponential scale of a sequence of probability measures $\{\nu_n\}_{n\geq 1}$. This behaviour is governed by a rate function. We make this precise in the following definition.

\begin{definition}
Let $\{\nu_n\}_{n\geq1}$ be a sequence of probability measures with values in a metric space $\cX$.
\begin{enumerate}
\item A \emph{rate function} is a lower semicontinuous function $I: \cX \to [0,\infty]$. A rate function is called \emph{good} if the level sets $\{x \in \cX| I(x) \leq c\}$ are compact for any $c \geq 0$.
\item The sequence $\{\nu_n\}_{n\geq1}$ satisfies the \emph{large deviation principle (LDP)} in $\cX$ with rate function $I$ if the following are satisfied:
\begin{enumerate}
\item (Upper bound) For any closed $F \subset \cX$
\[\label{eq:upper_bound}
\limsup_{n\to\infty} \frac1n\log\nu_n(F) \leq -\inf_{x\in F} I(x).
\]
\item (Lower bound) For any open $G \subset \cX$
\[
\liminf_{n\to\infty} \frac1n\log\nu_n(G) \geq -\inf_{x\in G} I(x).
\]
\end{enumerate}

\item The sequence $\{\nu_n\}_{n\geq1}$ is \emph{exponentially tight} if for every $\alpha > 0$ there exists a compact set $K_\alpha \subset \cX$ satisfying
\[
\limsup_{n\to\infty} \frac1n\log\nu_n(K_\alpha^c) < - \alpha.
\]
\end{enumerate}
\end{definition}

When a sequence of probability measures is exponentially tight, it is sufficient to know the upper bound of the large deviation principle only for compact sets. The upper bound then also immediately holds for all closed sets, see e.g. \cite[Section 1.2]{DZ98}.

\subsection{Riemannian geometry}\label{section:intro_geometry}

In this section we introduce the necessary notions from differential geometry, see for example \cite{Spi79} for a general introduction. We mainly focus on Riemannian geometry, for which we refer to \cite{Lee97} among others.\\

Let $(M,g)$ be a Riemannian manifold of dimension $N$. As usual, we denote by $TM$ the tangent bundle of $M$. For a point $x \in M$ we write $T_xM$ for the tangent space at $x$. Tangent vectors are usually denoted by $v$. A smooth assignment of tangent vectors to all points at $M$ is called a vector field, and the set of vector fields is denoted by $\Gamma(TM)$.\\

For $x\in M$ and $v,w \in T_xM$ we write the inner product as $\inp{v}{w}_{g(x)}$, where the subscript is omitted when the tangent space is understood. Given the inner product, we define the length of $v \in T_xM$ by its usual formula
\[
|v|_{g(x)} = \sqrt{\inp{v}{v}_{g(x)}}.
\]

Given a curve $\gamma:[a,b] \to M$, we define its length by
\[
L(\gamma;[a,b]) = \int_a^b|\dot\gamma(t)|\ud t.
\]
Using this length function, we define the Riemannian distance $d$ on $M$ as 
\begin{equation}\label{eq:Riemannian_distance}
d(x,y) := \inf\{L(\gamma)| \gamma:[a,b] \to M, \gamma(a) = x, \gamma(b) = y, \gamma \mbox{ piecewise smooth}\}.
\end{equation}

\subsubsection{Connection and parallel transport}

Associated to the Riemannian metric $g$ is a unique connection $\Nabla$, the \textit{Levi-Civita} connection, which is compatible with the metric and torsion free. 

A vector field $v(t)$ along a curve $\gamma(t)$ is called parallel if $D_tv(t) := \Nabla_{\dot\gamma(t)}v(t) = 0$. If the vector field $\dot\gamma(t)$ is parallel along $\gamma(t)$, then $\gamma$ is called a geodesic. It turns out that optimal paths for the distance between points in $M$ are geodesics for the Levi-Civita connection.

Equivalent to having a connection is having a notion of parallel transport. Given a curve $\gamma:[a,b] \to M$ and $v \in T_{\gamma(a)}M$, we can consider the unique solution $v(t)$ of the differential equation 
\[
\Nabla_{\dot\gamma(t)}v(t) = 0,  \qquad v(0) = v.
\]
This allows us to define a linear map 
\[
\tau_{\gamma(a)\gamma(t);\gamma}:T_{\gamma(a)}M \to T_{\gamma(t)}M
\]
by setting $\tau_{\gamma(a)\gamma(t);\gamma}v = v(t)$. The map $\tau_{\gamma(a)\gamma(t);\gamma}$ is called \emph{parallel transport} along $\gamma$. We omit the reference to the curve $\gamma$ when it is understood. Because $\Nabla$ is compatible with the Riemannian metric, parallel transport is in fact an isometry. 

Conversely, we can use parallel transport to compute covariant derivatives. To this end, let $v,w \in \Gamma(TM)$ be vector fields and $x \in M$. Let $\gamma$ be a curve with $\gamma(0) = x$ and $\dot\gamma(0) = v$. Then 
\[
\Nabla_vw(x) = \lim_{h\to0} \frac{\tau^{-1}_{x\gamma(h)}w(\gamma(h)) - w(x)}{h}.
\]

\subsubsection{Riemannian exponential map}

Given $x \in M$, define for every $v \in T_xM$ the geodesic $\gamma_v$ satisfying $\gamma_v(0) = x$ and $\dot\gamma_v(0) = v$. A priori, this geodesic does not exist for all time $t$. We say that the manifold $M$ is \textit{complete} if every such geodesic can be extended indefinitely. By the Hopf-Rinow theorem, this is equivalent to the completeness of $M$ as a metric space with the Riemannian distance $d$ defined in \eqref{eq:Riemannian_distance}.

We now define the \textit{Riemannian exponential map} $\Exp_x:\cE(x) \to M$ by setting $\Exp_xv = \gamma_v(1)$, where $\cE(x) \subset T_xM$ contains all $v \in T_xM$ for which $\gamma_v$ as above exists at least on $[0,1]$. If $M$ is complete, we have $\cE(x) = T_xM$. If additionally $M$ is simply connected, it holds that $\Exp_x$ is surjective.

However, due to curvature, the exponential map is not necessarily injective. For $x \in M$ we define the \textit{injectivity radius} $\iota(x) \in (0,\infty]$ as
\[
\iota(x) = \sup\{\delta > 0| \Exp_x \mbox{ is injective on } B(0,\delta)\}.
\]
Given a set $A \subset M$, the injectivity radius of $A$ is defined by
\begin{equation}\label{eq:injectivity_radius_set}
\iota(A) = \inf\{\iota(x)| x \in A\}.
\end{equation}
It can be shown (see e.g. \cite{Kli82}) that the map $x \mapsto \iota(x)$ is continuous on $M$. Consequently, for a compact set $K$ we have $\iota(K) > 0$.\\

The differential $\dd(\Exp_x)$ of the exponential map at $x$ is a linear map from $T(T_xM)$ into $TM$. Upon identifying $T_v(T_xM)$ with $T_xM$, we find that for any $v \in T_xM$ we have
\[
\dd(\Exp_x)_v:T_xM \to T_{\Exp_xv}M.
\]

\subsubsection{Jacobi fields}\label{section:Jacobi}

Let $\gamma:[0,1] \to M$ be a smooth curve. A \textit{variation} of $\gamma$ is a smooth map $\Gamma:(-\epsilon,\epsilon) \times [0,1] \to M$ such that $\Gamma(0,t) = \gamma(t)$ for all $t \in [0,1]$. Denoting by $s$ the first variable, the \textit{variational vector field} $V$ of $\Gamma$ is defined as 
\[
V(t) = \frac{\dd}{\dd s}\bigg|_{s=0} \Gamma(s,t) =: \partial_s\Gamma(0,t).
\]
Intuitively, $V$ measures the speed at which the curve $\gamma$ deforms. 

We denote by $D_t$ the covariant derivative along the curve $t \mapsto \Gamma(s,t)$, and similarly for $D_s$. Because the Levi-Civita connection is symmetric, we obtain the following symmetry lemma, see e.g. \cite[Lemma 6.3]{Lee97}.

\begin{lemma}[Symmetry lemma]\label{lemma:symmetry}
Let $\gamma:[0,1]\to M$ be a smooth curve and $\Gamma:(-\epsilon,\epsilon) \times [0,1] \to M$ a variation of $\gamma$. If $M$ is equipped with the Levi-Civita connection, then
\[
D_s\partial_t\Gamma(s,t) = D_t\partial_s\Gamma(s,t).
\]
\end{lemma}

Now suppose $\gamma:[0,1] \to M$ is a geodesic. Let $\Gamma:(-\epsilon,\epsilon) \times [0,1] \to M$ be a variation of $\gamma$ such that for any $s \in (-\epsilon,\epsilon)$, the curve $t \mapsto \Gamma(s,t)$ is a geodesic. We call $\Gamma$ a \emph{variation of geodesics}, and the corresponding variational vector field is called a \textit{Jacobi field} along $\gamma$.\\

It is possible to derive a second order differential equation satisfied by Jacobi fields. For this, we need to introduce the Riemann curvature endomorphism. The \textit{Riemann curvature endomorphism} measures the commutativity of second order covariant derivatives of a vector field. More precisely, it is a map $R: \Gamma(TM) \times \Gamma(TM) \times \Gamma(TM) \to \Gamma(TM)$ defined by
\[
R(v,w)u = \Nabla_v\Nabla_wu - \Nabla_w\Nabla_vu - \Nabla_{[v,w]}u,
\]
where $[v,w] =  vw - wv$ is the commutator of the vector fields $v$ and $w$. \\

One can show (see e.g. \cite[Theorem 10.2]{Lee97} or \cite[Section 10.1]{Fra04}) that a Jacobi field $J(t)$ along a geodesic $\gamma$ satisfies
\begin{equation}\label{eq:jacobi}
D_t^2J(t) + R(J(t),\dot\gamma(t))\dot\gamma(t) = 0,
\end{equation}
where $R$ denotes the Riemann curvature endomorphism. Equation \eqref{eq:jacobi} is called the Jacobi equation.\\

If $J(0) = 0$ and $\dot J(0)$ is given, a Jacobi field along a geodesic $\gamma$ satisfying these conditions is
\[
J(t) = \dd(\Exp_{\gamma(0)})_{t\dot\gamma(0)}(t\dot J(0)).
\]
This can be seen by considering the variation $\Gamma(t,s) = \Exp_{\gamma(0)}(t(\dot\gamma(0) + s\dot J(0)))$. The condition that $J(0) = 0$ indicates that all geodesics in the variation start in the same point.

In Euclidean space, this Jacobi field reduces to $J(t) = t\dot J(0)$, which is indeed the variation field of the variation $\Gamma(t,s) = \gamma(0) + t(\dot\gamma(0) + s\dot J(0))$. \\

We conclude this section by collecting some properties of Jacobi fields that we need later on. We include the arguments for the reader's convenience.

\begin{proposition}\label{prop:inner_product_jacobi}
Let $\gamma:[0,1] \to M$ be a geodesic and $J(t)$ a Jacobi field along $\gamma$. Then
\[
\inp{J(t)}{\dot\gamma(t)} = t\inp{\dot J(0)}{\dot\gamma(0)} + \inp{J(0)}{\dot\gamma(0)}
\]
for all $t \in [0,1]$.
\end{proposition}
\begin{proof}
Define $f(t) = \inp{J(t)}{\dot\gamma(t)}$. Then 
\[
f'(t) = \inp{D_t J(t)}{\dot\gamma(t)} + \inp{J(t)}{D_t\dot\gamma(t)} = \inp{D_t J(t)}{\dot\gamma(t)},
\]
because $\gamma$ is a geodesic. We are done once we show that $f''(t) = 0$. For this, notice that, using \eqref{eq:jacobi}
\[
f''(t) = \inp{D_t^2 J(t)}{\dot\gamma(t)} = -\inp{R(J(t),\dot\gamma(t))\dot\gamma(t)}{\dot\gamma(t)} = 0.
\]
Here, the last step follows from the symmetry properties of the Riemann curvature tensor.
\end{proof}

\begin{proposition}\label{prop:norm_jacobi}
Let $\gamma:[0,1] \to M$ be a geodesic and $J(t)$ a Jacobi field along $\gamma$. For every $t \in [0,1]$ there exists $\xi_t \in (0,t)$ such that
\[
|\dot J(t)| = |\dot J(0)| - t\frac{1}{|\dot J(\xi_t)|}\inp{R(J(\xi_t),\dot\gamma(\xi_t))\dot\gamma(\xi_t)}{\dot J(\xi_t)}.
\]
\end{proposition}
\begin{proof}
Define $f(t) = |\dot J(t)|$. We have
\begin{align*}
f'(t)
&=
\frac{1}{|\dot J(t)|}\inp{\ddot J(t)}{\dot J(t)}
\\
&=
-\frac{1}{|\dot J(t)|}\inp{R(J(t),\dot\gamma(t))\dot\gamma(t)}{\dot J(t)}.
\end{align*}
The statement now follows from the mean-value theorem.
\end{proof}


\section{Geodesic random walks}\label{section:geodesic_random_walks}

In order to generalize Cram\'er's theorem to the setting of Riemannian manifolds, we first need to introduce the appropriate analogue of the sequence $\{\frac1n\sum_{i=1}^n X_i\}_{n\geq 0}$ for a sequence of increments $\{X_n\}_{n\geq1}$. In order to do this, we introduce geodesic random walks, following the construction in \cite{Jor75}. Finally, we generalize the notion of identically distributed increments to geodesic random walks and characterize it using log moment generating functions.

\subsection{Definition of geodesic random walks}

We start by defining a \textit{geodesic random walk} $\{\Ss_n\}_{n\geq0}$ on $M$ with increments $\{X_n\}_{n\geq 1}$. For this we need to generalize how to add increments together. This is achieved by using the Riemannian exponential map. Because the space variable determines in which tangent space the increment should be, we have to define the random walk recursively, which is the main difficulty in the definition below.

\begin{definition}
Fix $x_0$ in $M$. A pair $(\{\Ss_n\}_{n\geq 0},\{X_n\}_{n\geq 1})$ is called a \emph{geodesic random walk} with increments $\{X_n\}_{n\geq1}$ and started at $x_0$ if the following hold:
\begin{enumerate}
\item $\Ss_0 = x_0$,
\item $X_{n+1} \in T_{\Ss_n}M$ for all $n \geq 0$,
\item $\Ss_{n+1} = \Exp_{\Ss_n}(X_{n+1})$ for all $n \geq 0$.
\end{enumerate}
\end{definition}

In what follows, the sequence $\{X_n\}_{n\geq 1}$ of increments will usually be omitted and we simply write that $\{\Ss_n\}_{n\geq0}$ is a geodesic random walk with increments $\{X_n\}_{n\geq 1}$. 

Note that in the above definition, we fix nothing about the distribution of the increments $\{X_n\}_{n\geq1}$. The distribution is allowed to depend both on the space variable, as well as on time.

For $M = \RR^N$, the Riemannian exponential map can be identified with addition, i.e., $\Exp_x(v) = x + v$. Hence, a geodesic random walk in $\RR^N$ reduces to the usual random walk, i.e. $\Ss_n = \sum_{i=1}^n X_i$.\\

Next, we introduce the concept of time-homogeneous increments for geodesic random walks. For this, we need to fix the distribution of the increments independent of the time variable. Because the increments can take values in different tangent spaces, we need a collection of measures $\{\mu_x\}_{x\in M}$ such that $\mu_x$ is a probability measure on $T_xM$ for every $x \in M$. We denote the set of probability measures on $T_xM$ by $\cP(T_xM)$. We have the following definition.

\begin{definition}
Let $\{\Ss_n\}_{n\geq0}$ be a geodesic random walk with increments $\{X_n\}_{n\geq 1}$ and started at $x_0$. Let $\{\mu_x\}_{x\in M}$ be a collection of measures such that $\mu_x \in \cP(T_xM)$ for every $x \in M$. We say the random walk $(\{\Ss_n\}_{n\geq0},\{X_n\}_{n\geq 1})$ is \emph{compatible} with the collection $\{\mu_x\}_{x\in M}$ if $X_{n+1} \sim \mu_{\Ss_n}$ for every $n \geq 0$.  
\end{definition}

Essentially, the collection of measures provides the distributions for the increments of the geodesic random walk. Because the collection of measures is independent of $n$, the increments are time-homogeneous.\\

Next, we want to define what it means for the increments of a geodesic random walk to be independent. Because the distribution of increment $X_{n+1}$ depends on $\Ss_n$, we have that $X_{n+1}$ is in general not independent of $\cA_n = \sigma(\{X_1,\ldots,X_n\})$ in the usual sense. However, this dependence is purely geometric, as $\Ss_n$ simply determines in which tangent space we have to choose $X_{n+1}$. If this is the only dependence of $X_{n+1}$ on $\cA_n$, we say the increments of $\{\Ss_n\}_{n\geq 0}$ are independently distributed. We make this precise in the following definition.

\begin{definition}\label{definition:independent_increments}
 Let $\{\mu_x\}_{x\in M}$ be a collection of measures such that $\mu_x \in \cP(T_xM)$ for every $x \in M$. Let $\{\Ss_n\}_{n\geq 0}$ be a geodesic random walk with increments $\{X_n\}_{n\geq1}$, compatible with $\{\mu_x\}_{x\in M}$. For every $n \geq 1$, define the $\sigma$-algebra $\Ff_n$ by
\[
\Ff_n = \sigma(\{(\Ss_0,X_1),\ldots,(\Ss_{n-1},X_n)\}).
\]
We say the increments of $\{\Ss_n\}_{n\geq0}$ are \emph{independent}, if for every $n \geq 1$ and all bounded, continuous functions $f:M^n \to \RR$ we have
\[
\EE(f(X_1,\ldots,X_n)|\Ff_{n-1}) = \int_{T_{S_{n-1}}M} f(X_1,\ldots,X_{n-1},v) \mu_{\Ss_{n-1}}(\dd v).
\]
\end{definition}

\begin{remark}
Because $\Ss_n = \Exp_{\Ss_{n-1}}X_n$, we have that $\Ss_n$ is $\Ff_n$-measurable. Consequently, we have $\sigma(\{\Ss_0,\ldots,\Ss_n\}) \subset \Ff_n$. However, equality need not hold. Indeed, if the Riemannian exponential map $\Exp_x$ is not injective, one cannot retrieve the increments $X_1,\ldots,X_n$ from $\Ss_0,\ldots,\Ss_n$.
\end{remark}

\begin{remark}
Let $\{\mu_x\}_{x\in M}$ be a collection of measures such that $\mu_x \in \cP(T_xM)$ for all $x \in M$. Let $\{\Ss_n\}_{n\geq0}$ be a geodesic random walk with increments $\{X_n\}_{n\geq 1}$ compatible with $\{\mu_x\}_{x\in M}$. Suppose furthermore that the increments are independent. Then $\{\Ss_n\}_{n\geq0}$ is a time-homogeneous, discrete time Markov process on $M$ with transition operator
\[
Pf(x) = \EE(f(\Ss_1)|\Ss_0 = x) = \int_{T_xM} f(\Exp_x(v)) \mu_x(\dd v).
\]
This is the point of view taken in \cite{KRV18}.
\end{remark}

\subsubsection{Rescaled geodesic random walks}

In Euclidean space, one commonly encounters rescaled versions of a random walk, for example for laws of large numbers and central limit theorems. On a general manifold, this rescaling cannot be achieved by multiplication. 

Before we define the appropriate analogue of $\{\frac1n\sum_{i=1}^n X_i\}_{n\geq0}$, we first need to define how to rescale a geodesic random walk by a factor $\alpha > 0$ independent of $n$. Note that in Euclidean space we can write $\alpha\sum_{i=1}^n X_i = \sum_{i=1}^n (\alpha X_i)$. This shows that we should rescale the increments of the random walk, which is possible in a manifold, because the increments are tangent vectors.

\begin{definition}

Fix $x_0$ in $M$ and $\alpha > 0$. A pair $(\{(\alpha*\Ss)_n\}_{n\geq 0},\{X_n\}_{n\geq 1})$ is called an $\alpha$-\emph{rescaled geodesic random walk} with increments $\{X_n\}_{n\geq1}$ and started at $x_0$ if the following hold:
\begin{enumerate}
\item $(\alpha*\Ss)_0 = x_0$,
\item $X_{n+1} \in T_{(\alpha*\Ss)_n}M$ for all $n \geq 0$,
\item $(\alpha*\Ss)_{n+1} = \Exp_{(\alpha*\Ss)_n}(\alpha X_{n+1})$ for all $n \geq 0$.
\end{enumerate}
\end{definition}

As with geodesic random walks, we will often omit the sequence of increments and simply write that $\{(\alpha*\Ss)_n\}_{n\geq0}$ is an $\alpha$-rescaled geodesic random walk with increments $\{X_n\}_{n\geq1}$. 

Note that an $\alpha$-rescaled geodesic random walk can itself be considered as a geodesic random walk. Indeed, if $(\alpha*\Ss)_n$ is an $\alpha$-rescaled geodesic random walk with increments $\{X_n\}_{n\geq1}$, then it is a geodesic random walk with increments $\{\alpha X_n\}_{n\geq 1}$.

As for geodesic random walks, we say that an $\alpha$-rescaled geodesic random walk $\{(\alpha*\Ss)_n\}_{n\geq0}$ with increments $\{X_n\}_{n\geq1}$ is compatible with a collection of probability measures $\{\mu_x\}_{x\in M}$ if $X_{n+1} \sim \mu_{(\alpha*\Ss)_n}$ for every $n \geq 0$. It follows that when considered as geodesic random walk, $\{(\alpha*\Ss_n)\}_{n\geq 0}$ is compatible with the collection of measures $\{\mu_x^\alpha\}_{x\in M}$ given by
\[
\mu_x^\alpha = \mu_x \circ m_\alpha^{-1}
\] 
where $m_\alpha:T_xM \to T_xM$ denotes multiplication by $\alpha$, i.e., $m_\alpha(v) = \alpha v$.

\subsubsection{Empirical average process} \label{section:empirical_average}

We conclude this section by introducing the analogue of the sequence of empirical averages $\{\frac1n\sum_{i=1}^n X_i\}_{n\geq 0}$ for a sequence $\{X_n\}_{n\geq 1}$ of random variables.

Fix $x_0 \in M$ and let $\{\mu_x\}_{x\in M}$ be a collection of measures such that $\mu_x \in \cP(T_xM)$ for all $x \in M$. For every $n \geq 1$, let $\{(\frac1n*\Ss)_j\}_{j\geq0}$ be a $\frac1n$-rescaled geodesic random walk started at $x_0$ with  increments $\{X_j^n\}_{j\geq 1}$, compatible with the measures $\{\mu_x\}_{x\in M}$. By considering the diagonal elements of $\{(\frac1n*\Ss)_j\}_{n\geq1,j\geq 0}$, we obtain for every $n \geq 1$ a random variable $(\frac1n*\Ss)_n$ in $M$.  If we now set the initial value of the sequence $\{(\frac1n*\Ss)_n\}_{n\geq 0}$ to be $x_0$, we obtain the Riemannian analogue of the sequence $\{\frac1n\sum_{i=1}^n X_i\}_{n\geq 0}$. We refer to this process as the \emph{empirical average process} started at $x_0$ compatible with the collection of measures $\{\mu_x\}_{x\in M}$.

\subsection{Identically distributed increments}

For our purposes, we also need a notion of identically distributed increments. In general, the increments of a geodesic random walk do not live in the same tangent space. In order to overcome this problem, we use parallel transport to identify tangent spaces. Because the identification via parallel transport depends on the curve along which the vectors are transported, we need to make the following definition.

\begin{definition}\label{definition:consistency}
Let $\{\mu_x\}_{x\in M}$ be a collection of measures such that $\mu_x \in \cP(T_xM)$ for all $x \in M$. Let $\{\Ss_n\}_{n\geq 0}$ be a geodesic random walk with increments $\{X_n\}_{n \geq 1}$, compatible with $\{\mu_x\}_{x\in M}$. We say the increments $\{X_n\}_{n\geq1}$ are \emph{identically distributed} if the measures satisfy the following \emph{consistency property}: for any $y,z \in M$ and any smooth curve $\gamma:[a,b] \to M$ with $\gamma(a) = y$ and $\gamma(b) = z$ we have
\[
\mu_z = \mu_y \circ \tau_{yz;\gamma}^{-1}.
\]
\end{definition}

By the transitivity property of parallel transport, one can equivalently define the consistency property to hold for all piecewise smooth curves.

Note that in Euclidean space, our definition of independent increments implies that the measures are independent of the space variable, because parallel transport is the identity map. Hence, our definition reduces to the usual one, as we obtain that every increment has some fixed distribution $\mu$.

Because parallel transport is an isometry, we can use distributions with spherical symmetry to construct a family of measures $\{\mu_x\}_{x\in M}$ satisfying Definition \ref{definition:consistency}. We refer to \cite[Section 4]{KRV18} for the details and more specific examples.\\

The consistency property  in Definition \ref{definition:consistency} may also be characterised by a consistency assumption for the corresponding log-moment generating functions $\Lambda_x:T_xM \to \RR$ of $\mu_x$ given by
\[
\Lambda_x(\lambda) = \log \int_{T_xM} e^{\inp{\lambda}{v}} \mu_x(\dd v).
\]
This is recorded in the following proposition, which can be found in \cite[Section 4]{KRV18}. 

\begin{proposition}\label{prop:logmgf}
Let $\{\mu_x\}_{x\in M}$ be a collection of measures such that $\mu_x \in \cP(T_xM)$ for every $x \in M$. Assume that $\Lambda_x(\lambda) < \infty$ for all $x \in M$ and all $\lambda \in T_xM$. The following are equivalent:
\begin{enumerate}[(a)]
\item \label{prop:logmgf_consistency} The collection $\{\mu_x\}_{x\in M}$ satisfies the consistency property in Definition \ref{definition:consistency}.
\item \label{prop:logmgf_laplace_transform} For all $x,y \in M$ and all smooth curves $\gamma:[a,b] \to M$ with $\gamma(a) = x$ and $\gamma(b) = y$ and for all $\lambda \in T_xM$ we have
\[
\Lambda_x(\lambda) = \Lambda_y(\tau_{xy;\gamma}\lambda).
\]
\end{enumerate}
\end{proposition} 

The \emph{Legendre transform} $\Lambda_x^*:T_xM \to \RR$ of $\Lambda_x$ is defined by
\[
\Lambda_x^*(v) := \sup_{\lambda \in T_xM} \inp{\lambda}{v} - \Lambda_x(\lambda).
\]
If the collection of log-moment generating functions $\{\Lambda_x\}_{x\in M}$ satisfies the consistency property in \eqref{prop:logmgf_laplace_transform} of the above proposition, then so does the collection $\{\Lambda_x^*\}_{x\in M}$ of their Legendre transforms.


\section{Sketch of the proof of Cram\'er's theorem for Riemannian manifolds}\label{section:approach_proof}

In this section we provide a sketch of the proof of Cram\'er's theorem for geodesic random walks and stress what observations and properties are important to make the proof work. Before we get to this, let us first state the exact theorem we wish to prove.

\subsection{Statement of Cram\'er's theorem}\label{section:statement_cramer}

Cram\'er's theorem is concerned with the large deviations for the empirical average process $\{(\frac1n*\Ss)_n\}_{n\geq1}$ with independent, identically distributed increments.

Along with the large deviation principle, we need to identify the rate function. In Euclidean space, the rate function is given by
\[
I(x) = \Lambda^*(x),
\]
the Legendre transform of the log moment generating function of an increment. Note here that one can consider the vector $x$ as the tangent vector of the straight line from the origin to the point $x$. Using this viewpoint, the analogue of the rate function in the Riemannian setting should be
\[
I(x) = \inf\{\Lambda_{x_0}^*(v)| \Exp_{x_0}v = x\}. 
\]
Here, we have to take the infimum, because the Riemannian exponential map is not necessarily injective, i.e., there may be more than one geodesic connecting $x_0$ and $x$. We will show that this is indeed the correct rate function, as collected in the following theorem.

\begin{theorem}[Cram\'er's theorem for Riemannian manifolds]\label{theorem:Cramer}
Let $(M,g)$ be a complete Riemannian manifold. Fix $x_0 \in M$ and let $\{\mu_x\}_{x\in M}$ be a collection of measures such that $\mu_x \in \cP(T_xM)$ for all $x \in M$.  For every $n \geq 1$, let $\{(\frac1n*\Ss)_j\}_{j\geq0}$ be a $\frac1n$-rescaled geodesic random walk started at $x_0$ with independent increments $\{X_j^n\}_{j\geq1}$, compatible with $\{\mu_x\}_{x\in M}$. Let $\{(\frac1n*\Ss)_n\}_{n \geq 0}$ be the associated empirical average process started at $x_0$.  Assume the increments are bounded and have expectation 0. Assume furthermore that the collection $\{\mu_x\}_{x\in M}$ satisfies the consistency property in Definition \ref{definition:consistency}. Then $\{(\frac1n*\Ss)_n\}_{n\geq0}$ satisfies in $M$ the LDP with good rate function
\begin{equation}\label{eq:rate_cramer}
I_M(x) = \inf\{\Lambda_{x_0}^*(v)| \Exp_{x_0}v = x\}
\end{equation}

\end{theorem}

Due to geometrical influences, which become apparent when sketching the proof, we prove Cram\'er's theorem only in the case when the increments are bounded. This allows for a less technical proof of the theorem, but nevertheless introduces all geometrical obstructions that have to be dealt with. The details of the proof can be found in Section \ref{section:Cramer_bounded}.\\

Like in the Euclidean setting, we prove Cram\'er's theorem for geodesic random walks by separately proving the upper and lower bound for the large deviation principle of $\{(\frac1n*\Ss)_n\}_{n\geq0}$. In Section \ref{section:sketch_upper_bound} we give an overview of the steps one needs to take to prove the upper bound, while in Section \ref{section:sketch_lower_bound} we sketch how to prove the lower bound.

\subsection{Sketch of the proof of the upper bound}\label{section:sketch_upper_bound}

In the Euclidean case, one proves the upper bound in Cram\'er's theorem by using Chebyshev's inequality. More precisely, the key step is to show that for $\Gamma \subset \RR^d$ compact one has (see e.g. \cite{Hol00,DZ98})
\[
\limsup_{n\to\infty} \frac1n\log\PP\left(\frac1nS_n \in \Gamma\right) \leq -\inf_{x\in\Gamma} \sup_{\lambda \in \RR^d} \left\{\inp{\lambda}{x} - \limsup_{n\to\infty} \frac1n\log\EE\left(e^{n\inp{\lambda}{\frac1nS_n}}\right)\right\}.
\]
The upper bound is then extended to all closed sets by proving exponential tightness. The idea is to follow a similar procedure in the Riemannian case. However, because $(\frac1n*\Ss)_n$ is $M$-valued, its moment generating function is not defined.

\subsubsection{Step 1: Analogue of the moment generating function $\EE(e^{n\inp{\lambda}{\frac1nS_n}})$}

To overcome the problem of not having a moment generating function of $(\frac1n*\Ss)_n$, we want to identify points in $M$ with tangent vectors in $T_{x_0}M$. For this we use the Riemannian exponential map. However, this map is not necessarily injective. Hence, we first assume that for each $n \geq 1$, the $\frac1n$-rescaled geodesic random walk stays within the injectivity radius $\iota(x_0)$ of its initial point $x_0$ up to time $n$. Consequently, because $\Exp_{x_0}$ is injective on $B(0,\iota(x_0)) \subset T_{x_0}M$, we can uniquely define $v_k^n \in T_{x_0}M$ satisfying $|v_k^n| < \iota(x_0)$ and 
\[
\Exp_{x_0}^{-1}(v_k^n) = \left(\frac1n*\Ss\right)_k.
\]

Ideally, we would like to prove the large deviation principle for $\{(\frac1n*\Ss)_n\}_{n\geq0}$ by proving the large deviation principle for $\{v_n^n\}_{n\geq0}$ in $T_{x_0}M$ and then apply the contraction principle (see e.g. \cite[Chapter 4]{DZ98}) with the continuous function $\Exp_{x_0}$. For this to work, we would need to show that
\begin{equation}\label{eq:limit_logmgf}
\lim_{n\to\infty} \frac1n\log\EE(e^{n\inp{\lambda}{v_n^n}}) = \Lambda_{x_0}(\lambda).
\end{equation}
Unfortunately, using the estimate for $\EE(e^{n\inp{\lambda}{v_n^n}})$ found in Step 2 as explained below, we are only able to show that
\begin{equation}\label{eq:sketch_upper}
\limsup_{n\to\infty} \frac1n\log\EE(e^{n\inp{\lambda}{v_n^n}}) \leq \Lambda_{x_0}(\lambda) + C|\lambda|
\end{equation}
and likewise
\begin{equation}\label{eq:sketch_lower}
\liminf_{n\to\infty} \frac1n\log\EE(e^{n\inp{\lambda}{v_n^n}}) \geq \Lambda_{x_0}(\lambda) - C|\lambda|,
\end{equation}
where the constant only depends on the curvature and the uniform bound of the increments.

\subsubsection{Step 2: Upper bound for the moment generating function of $v_n^n$}

In $\RR^d$ we simply have $v_n^n = \frac1n\sum_{i=1}^n X_i$ and hence its moment generating function is given by
\[
\EE(e^{n\inp{\lambda}{v_n^n}}) = \prod_{i=1}^n \EE(e^{\inp{\lambda}{X_i}}) = \EE(e^{\inp{\lambda}{X_1}})^n. 
\]
Here we use the fact that we can write $v_k^n = v_{k-1}^n + \frac1nX_k$. This fails in the Riemannian setting, which results in the fact that we can only estimate $\EE(e^{n\inp{\lambda}{v_n^n}})$ as mentioned above in \eqref{eq:sketch_upper} and \eqref{eq:sketch_lower}.\\

In a Riemannian manifold we replace the identity $v_k^n = v_{k-1}^n + \frac1nX_k$ by the Taylor expansion of $\Exp_{x_0}^{-1}$ (see Section \ref{section:Taylor_inverse_exponential}, Proposition \ref{proposition:inverse_exp}). This results in
\begin{equation}\label{eq:estimate_Exp_inverse}
v_k^n = v_{k-1}^n + \frac1n\dd(\Exp_{x_0})^{-1}_{v_{k-1}^n}X_k^n + \Oo\left(\frac1{n^2}\right).
\end{equation}
Here one needs to be careful that the constant in the error term may depend on curvature properties of the manifold around $(\frac1n*\Ss)_{k-1}$. Because we assume the increments are uniformly bounded, there exists a compact set $K \subset M$ such that for all $n \geq 1$ and all $0 \leq j \leq n$ we have $(\frac1n*\Ss)_j \in K$. This allows us to control the constant in the error term. 

However, the problem arises that this expression does not yet allow us to use the assumption that the increments of the geodesic random walk are identically distributed, which essentially means that the distribution of the increments is invariant under parallel transport.

Consequently, we need to argue that $\dd(\Exp_{x_0})_{v_{k-1}^n}^{-1}$ can be approximated well enough by parallel transport. It turns out there exists a constant $C > 0$ such that
\begin{equation}\label{eq:comparison_parallel_differential}
|\dd(\Exp_{x_0})_{v_{k-1}^n}^{-1}X_k^n - \tau^{-1}_{x_0\frac1n\Ss_{k-1}} X_k^n| \leq C|v_{k-1}^n|^2|X_k^n|,
\end{equation}
see Section \ref{section:Taylor_differential_exponential} for details, in particular Corollary \ref{corollary:inverse_exp_parallel}. By the same reasoning as before, the constant $C$ may be controlled independent of $k$.

Combining \eqref{eq:estimate_Exp_inverse} and \eqref{eq:comparison_parallel_differential} and using that $v_n^n = \sum_{k=1}^n v_k^n - v_{k-1}^n$, we have 
\begin{equation}\label{eq:sum_parallel_transport}
\left|v_n^n - \frac1n\sum_{k=1}^n \tau^{-1}_{x_0\frac1n\Ss_{k-1}} X_k^n\right| \lesssim \frac1n + 1
\end{equation}

Consequently, using the Cauchy-Schwarz inequality, we find
\begin{equation}\label{eq:estimate_mgf}
\begin{aligned} 
\EE(e^{n\inp{\lambda}{v_n^n}}) 
&\leq
e^{C|\lambda|}e^{nC|\lambda|}\EE\left(e^{\sum_{i=1}^n\inp{\lambda}{\tau^{-1}_{x_0\frac1n\Ss_{n-1}} X_k^n}}\right)
\\
&=
e^{C|\lambda|}e^{nC|\lambda|}\EE\left(e^{\inp{\lambda}{X_1}}\right)^n.
\end{aligned}
\end{equation}
Here, the last line uses that the increments are independent and identically distributed. From this it follows that
\[
\limsup_{n\to\infty} \frac1n\log \EE(e^{n\inp{\lambda}{v_n^n}}) \leq C|\lambda| + \Lambda_{x_0}(\lambda),
\]
so that
\[
\limsup_{n\to\infty} \frac1n\log \PP(v_n^n \in F) \leq -\inf_{v\in F} \sup_{\lambda \in T_{x_0}M} \{\inp{\lambda}{v} - \Lambda_{x_0}(\lambda) - C|\lambda|\}.
\]
It remains to get rid of the $C|\lambda|$ term. In the next step we show how to reduce the order $n$ term in the upper bound in \eqref{eq:estimate_mgf}, so that we can still use the above estimating procedure to obtain the upper bound of the large deviation principle for $\{(\frac1n*\Ss)_n\}_{n\geq0}$.

\subsubsection{Step 3: Reducing the upper bound in Step 2 by splitting the random walk in pieces}

The problematic factor in estimate \eqref{eq:estimate_mgf} arises from the replacement of the differential of the exponential map with parallel transport as done in Step 2. This error depends on $|v_k^n|$, i.e., the distance from $x_0$ to $(\frac1n*\Ss)_k$. Note that in Step 2, we simply estimated $|v_k^n|$ uniformly in $k$. However, if we write $r$ for the uniform bound on the increments, we actually have $|v_k^n| \leq \frac knr$. Consquently, we can reduce the upper bound if the amount of steps for which we need to compare parallel transport and the differential of the exponential map becomes smaller. 

To do this, the idea is to cut the random walk in finitely many pieces, say $m$, each consisting of (roughly) $m^{-1}n$ steps. We can then consider each of these pieces as separate random walks which we need to identify with a vector in some tangent space. In the end, we can then let the amount of pieces tend to infinity by considering the limit $m \to \infty$, so that the part of the upper bound which we want to reduce vanishes entirely.

More precisely, fix $m \in \NN$, and define for $l = 0,\ldots,m-1$ the indices $n_l = l\lfloor m^{-1}n\rfloor$ and set $n_m = n$. This divides the random walk in $m$ pieces, where a piece starts in $(\frac1n*\Ss)_{n_l}$ and consists of $\lfloor m^{-1}n \rfloor$ increments. Now recall there is a compact set $K \subset M$ such that for all $n$ and all $0 \leq j \leq n$ we have $(\frac1n*\Ss)_j \in K$. Because $\iota(K) > 0$, we can choose $m$ sufficiently large, such that for all $n$, all $l = 1,\ldots,m$ and all $k = 1,\ldots, \lfloor m^{-1}n \rfloor$ we have 
\[
\left(\frac1n*\Ss\right)_{n_{l-1}+k} \in B\left(\left(\frac1n*\Ss\right)_{n_{l-1}},\iota\left(K\right)\right).
\]

Consequently, we may follow the same procedure as in Step 1, so that for every $l = 1,\ldots, m$ and every $k = 1,\ldots, \lfloor m^{-1}n \rfloor$ we can uniquely define $\tilde v_k^{n,m,l} \in T_{(\frac1n*\Ss)_{n_{l-1}}}M$ such that
\[
\tilde v_k^{n,m,l} \in \Exp_{(\frac1n*\Ss)_{n_{l-1}}}^{-1}\left(\left(\frac1n*\Ss\right)_{n_{l-1}+k}\right)
\]
and $|\tilde v_k^{n,m,l}| < \iota((\frac1n*\Ss)_{n_{l-1}})$. Finally, we define $v_k^{n,m,l} \in T_{x_0}M$ by
\[
v_k^{n,m,l} = \tau_{x_0(\frac1n*\Ss)_{n_{l-1}}}^{-1}\tilde v_k^{n,m,l},
\]
where the parallel transport can be taken along any path connecting $x_0$ and $\left(\frac1n*\Ss\right)_{n_{l-1}}$, as long as it is measurable with respect to $\Ff_{n_{l-1}} = \sigma(X_1,\ldots,X_{n_{l-1}})$. 

This associates to $(\frac1n*\Ss)_n \in M$ a tuple
\[
\left(v_{\lfloor m^{-1}n\rfloor}^{n,m,1},\ldots,v_{\lfloor m^{-1}n\rfloor}^{n,m,m}\right) \in (T_{x_0}M)^m.
\]
Following the procedure in Step 2, apart from some technical details, we find
\[
\limsup_{n\to\infty} \frac1n\log \EE\left(e^{n\inp{\lambda}{v_{\lfloor m^{-1}n\rfloor}^{n,m,l}}}\right) \leq C|\lambda|\frac1{m^3} + \frac1m\Lambda_{x_0}(\lambda), 
\]
 for all $\lambda \in T_{x_0}M$. From here it is possible to show that
\[
\limsup_{n\to\infty} \frac1n\log \EE\left(e^{n\sum_{l=1}^m\inp{\lambda_l}{v_{\lfloor m^{-1}n\rfloor}^{n,m,l}}}\right) \leq C\frac1{m^3}\sum_{l=1}^m|\lambda_l| + \frac1m\sum_{l=1}^m\Lambda_{x_0}(\lambda_l)
\]
for all $(\lambda_1,\ldots,\lambda_m) \in (T_{x_0}M)^m$. Consequently, we find that
\begin{align*}
&\limsup_{n\to\infty} \frac1n\log \PP\left(\left(v_{\lfloor m^{-1}n\rfloor}^{n,m,1},\ldots,v_{\lfloor m^{-1}n\rfloor}^{n,m,m}\right)  \in F\right)
\\
&\leq
-\inf_{(v_1,\ldots,v_m) \in F} \frac1m\sum_{l=1}^m \sup_{\lambda \in T_{x_0}M} \{\inp{\lambda}{mv_l} - \Lambda_{x_0}(\lambda) - \frac1{m^2}C|\lambda|\}.
\end{align*}

\subsubsection{Step 4: Upper bound for the large deviation principle of $\{(\frac1n*\Ss)_n\}_{n\geq0}$}
To prove the large deviation upper bound for $\{(\frac1n*\Ss)_n\}_{n\geq0}$, we notice that the map sending $(v_{\lfloor m^{-1}n\rfloor}^{n,m,1},\ldots,v_{\lfloor m^{-1}n\rfloor}^{n,m,m})$ to $(\frac1n*\Ss)_n$ is continuous. Hence, if $F \subset M$ is closed, there exists a closed set $\tilde F \subset (T_{x_0}M)^m$ such that
\[
\PP\left(\left(\frac1n*\Ss\right)_n \in F\right) = \PP\left(\left(v_{\lfloor m^{-1}n\rfloor}^{n,m,1},\ldots,v_{\lfloor m^{-1}n\rfloor}^{n,m,m}\right)  \in \tilde F\right).
\]
From this it follows that
\begin{align*}
&\limsup_{n\to\infty} \frac1n\log\PP\left(\left(\frac1n*\Ss\right)_n \in F\right)
\\ 
&\leq 
-\inf_{(v_1,\ldots,v_m) \in \tilde F} \frac1m\sum_{l=1}^m \sup_{\lambda \in T_{x_0}M} \{\inp{\lambda}{v_l} - \Lambda_{x_0}(\lambda) - \frac1{m^2}C|\lambda|\}.
\end{align*}

Now note that for every $v \in \Exp_{x_0}^{-1}F$ we have that $(\frac1m v,\ldots,\frac1m v) \in \tilde F$. Furthermore, by convexity, the infimum in the upper bound is attained when all $v_i$ are equal. Consequently, the upper bound reduces to
\begin{align*}
&\limsup_{n\to\infty} \frac1n\log\PP\left(\left(\frac1n*\Ss\right)_n \in F\right)
\\
&\leq 
-\inf_{v \in \Exp_{x_0}^{-1}F} \sup_{\lambda \in T_{x_0}M} \{\inp{\lambda}{v} - \Lambda_{x_0}(\lambda) - \frac1{m^2}C|\lambda|\}.
\end{align*}
The desired upper bound now follows by considering the limit $m \to \infty$.

\subsection{Sketch of the proof of the lower bound} \label{section:sketch_lower_bound}

To prove the lower bound of the large deviation principle for $\{(\frac1n*\Ss)_n\}_{n\geq0}$, it suffices to show that if $G \subset M$ is open, then
\[
\liminf_{n\to\infty} \frac1n\log\PP\left(\left(\frac1n*\Ss\right)_n \in G\right) \geq -I_M(x),
\]
for all $x \in G$. Because $I_M(x) = \inf_{v \in \Exp_{x_0}^{-1}x}\Lambda_{x_0}^*(v)$, it is in fact sufficient to show that
\[
\liminf_{n\to\infty} \frac1n\log\PP\left(\left(\frac1n*\Ss\right)_n \in G\right) \geq -\Lambda_{x_0}^*(v)
\] 
for any $v \in \Exp_{x_0}^{-1}G$. Consequently, we again need to transfer the problem to the tangent space $T_{x_0}M$.

\subsubsection{Transfer to the tangent space $T_{x_0}M$} 

Similar to how estimate \eqref{eq:sum_parallel_transport} is derived, we find that
\begin{equation}\label{eq:parallel_transport_epsilon}
\left|v_{\lfloor m^{-1}n\rfloor}^n - \frac1n\sum_{k=1}^{\lfloor m^{-1}n\rfloor} \tau^{-1}_{x_0\frac1n\Ss_{k-1}} X_k^n\right| \lesssim \frac1{nm} + \frac1{m^3}.  
\end{equation}
Consequently, by choosing $m$ sufficiently large, we can get $v_{\lfloor m^{-1}n\rfloor}^n$ arbitrarily close to $\frac1n\sum_{k=1}^{\lfloor m^{-1}n\rfloor} \tau^{-1}_{x_0\frac1n\Ss_{k-1}} X_k^n$. Using the fact that the increments of the geodesic random walk are independent and identically distributed, we prove that $\sum_{k=1}^{\lfloor m^{-1}n\rfloor} \tau^{-1}_{x_0\frac1n\Ss_{k-1}} X_k^n$ is a sum of independent random variables, each distributed according to $\mu_{x_0}$. Consequently, by Cram\' er's theorem for vector spaces, for every $m \in \NN$ the sequence $\{ \frac1n\sum_{k=1}^{\lfloor m^{-1}n\rfloor} \tau^{-1}_{x_0\frac1n\Ss_{k-1}} X_k^n\}_{n\geq0}$ satisfies the large deviation principle in $T_{x_0}M$ with good rate function $I(v) = \frac1m\Lambda_{x_0}^*(mv)$. 

Putting everything together, after some technicalities, we find that if $\epsilon > 0$ is small enough, there exists a constant $c \in (0,1)$ such that for $m$ large enough
\begin{align}
&\liminf_{n\to\infty} \frac1n\log \PP(v_{\lfloor m^{-1}n\rfloor}^n \in B(v,\epsilon)) \nonumber
\\
&\geq
\liminf_{n\to\infty} \frac1n\log \PP\left(\frac1n\sum_{k=1}^{\lfloor m^{-1}n\rfloor} \tau^{-1}_{x_0\frac1n\Ss_{k-1}} X_k^n \in B(v,c\epsilon^2)\right) \label{eq:lower_bound_estimate}
\\
&\geq
\frac1m\Lambda_{x_0}^*(mv). \nonumber
\end{align}

In order to make use of this fact, we again need to divide the random walk in pieces, like in Step 3 in Section \ref{section:sketch_upper_bound}. Consequently, we again first identify $(\frac1n*\Ss)_n \in M$ with a tuple
\[
\left(\tilde v_{\lfloor m^{-1}n\rfloor}^{n,m,1},\ldots,\tilde v_{\lfloor m^{-1}n\rfloor}^{n,m,m}\right) \in T_{(\frac1n*\Ss)_{n_0}}M \times \cdots \times T_{(\frac1n*\Ss)_{n_m}}M.
\]
However, this time we need to be careful how we transport these vectors to $T_{x_0}M$. Indeed, we wish to do this in such a way that
\begin{equation}\label{eq:implication_epsilon}
\left(v_{\lfloor m^{-1}n\rfloor}^{n,m,1},\ldots,v_{\lfloor m^{-1}n\rfloor}^{n,m,m}\right) \in B(v,c\epsilon)^m \Rightarrow \left(\frac1n*\Ss\right)_n \in B(\Exp_{x_0}v,\epsilon).
\end{equation}

The key to making the correct choice is given by Proposition \ref{proposition:distance_geodesic_point}, which gives us control over how far geodesics can spread in a short time when starting in different points of the manifold. This result shows us how to choose the parallel transport based on the vector $v$, so that the curvature has only little effect. Essentially, one first transports a vector to an associated point on the geodesic with speed $v$ which connects $x_0$ and $x$. After that, one transports the vector along this geodesic to $x_0$. More precisely, we do the following:
\begin{enumerate}
\item Consider the geodesic $\gamma(t) = \Exp_{x_0}(tv)$ and for $i = 0,\ldots,m$ define the points $y_i = \gamma(\frac im)$. Note that $y_0 = x_0$. 
\item For every $i = 0,\ldots,m$ and every $x \in M$, choose a geodesic of minimal length connecting $y_i$ and $x$ and define $\tau_{y_ix}$ to be parallel transport along this geodesic.
\item Now define for $i = 1,\ldots,m$ the vector $v_{\lfloor m^{-1}n\rfloor}^{n,m,1} \in T_{x_0}M$ by
\[
v_{\lfloor m^{-1}n\rfloor}^{n,m,i} = \tau_{y_0y_i}^{-1}\tau_{y_i(\frac1n*\Ss)_{n_{i-1}}}^{-1}\tilde v_{\lfloor m^{-1}n\rfloor}^{n,m,i}
\]
\end{enumerate}

Now, given $G \subset M$ open, $x \in G$ and $v \in \Exp_{x_0}^{-1}x$, by \eqref{eq:implication_epsilon} we have
\[
\PP\left(\left(\frac1n*\Ss\right)_n \in G\right) \geq \PP\left(\left(v_{\lfloor m^{-1}n\rfloor}^{n,m,1},\ldots,v_{\lfloor m^{-1}n\rfloor}^{n,m,m}\right) \in B(v,c\epsilon^2)^m\right).
\]
Using this, an approach similar to the one used to obtain \eqref{eq:lower_bound_estimate}, also using that the increments are independent and identically distributed, gives us that
\begin{align*}
&\liminf_{n\to\infty} \frac1n\log\PP\left(\left(\frac1n*\Ss\right)_n \in G\right)
\\
&\geq 
\liminf_{n\to\infty} \frac1n\log\PP\left(\left(v_{\lfloor m^{-1}n\rfloor}^{n,m,1},\ldots,v_{\lfloor m^{-1}n\rfloor}^{n,m,m}\right) \in B(v,c\epsilon^2)^m\right)
\\
&\geq 
-\Lambda_{x_0}^*(v)
\end{align*}
which is as desired.


\section{Geometric results for the proof}\label{section:geometry}

This section focuses on geometric results needed for the proof of Cram\'er's theorem for geodesic random walks as sketched in Section \ref{section:approach_proof}. We obtain a Taylor expansion for the inverse Riemannian exponential map and estimate the residual term. Furthermore, we bound the difference between the differential of the Riemannian exponential map and parallel transport. This heavily relies on the theory of Jacobi fields, which have been introduced in Section \ref{section:Jacobi}. We conclude this section by proving how far geodesics can spread in a short time interval when starting in different points on the manifold.

\subsection{Taylor expansion of the inverse Riemannian exponential map} \label{section:Taylor_inverse_exponential}

The Riemannian exponential map $\Exp_x:T_xM \to M$ is a local diffeomorphism around 0. More precisely, it is a diffeomorphism between $B(0,\iota(x)) \subset T_xM$ and $\Exp_x(B(0,\iota(x)))$. Now suppose $\gamma(t)$ is a curve in $\Exp_x(B(0,\iota(x)))$. There exists a unique curve $w(t)$ in $B(0,\iota(x)) \subset T_xM$ such that $\Exp_xw(t) = \gamma(t)$. Our aim is to find a Taylor expansion for $w(t)$ around $t = 0$. Although this seems to be folklore, we also find a precise estimate of the residual term of the Taylor approximation. \\

Before we can do this, we first need two lemmas that will help us control the error term in the first order Taylor polynomial for the inverse of the Riemannian exponential map.

\begin{lemma}\label{lemma:uniform_bound_differential}
Let $K \subset M$ be compact and for any $x \in K$, let $K_x \subset T_xM$ be compact. Assume there exists a $C > 0$ such that $K_x \subset \overline{B(0,C)}$ for any $x \in K$. Then
\[
\sup_{x \in K} \sup_{v \in K_x} |\dd(\Exp_x)_v| < \infty
\]
\end{lemma}
\begin{proof}
Because the sets $K_x$ are uniformly bounded and $K$ is compact, it follows that
\[
\{(x,v) \in TM| x \in K, v \in K_x\}
\] 
is compact.

Now fix $x \in M$ and $v \in T_xM$. Because the Riemannian exponential map $\Exp:TM \to M\times M$ is continuous, there exists a neighbourhood $U \subset TM$ of $(x,v)$ such that 
\[
\Exp(U) \subset B(x,\iota(x)) \times  B(\Exp_xv,\iota(\Exp_xv)).
\]

Now for $(y,w) \in U$, and any $u \in T_{\Exp_xv}M$ we define
\[
F_{u,x,v}:(y,w) \mapsto |\tau_{\Exp_yw\Exp_xv}\dd(\Exp_y)_w\tau_{xy}u|
\]
where parallel transport is taken along the unique minimizing geodesic connecting the two points, which exists by the choice of $U$. We argue that $F_{u,x,v}$ is continuous for any $u \in T_{\Exp_xv}M$. By the choice of parallel transport, $\tau_{\Exp_yw\Exp_xv}$ and $\tau_{xy}$ are continuous. Furthermore, note that we can write 
\[
\dd(\Exp_y)_w\tilde u = \dd\Exp((y,w),\tilde u).
\]
Because $\Exp$ is smooth, it follows that $\dd\Exp$ is continuous. Consequently, $F_{u,x,v}$ is a composition of continuous maps, and hence continuous on $U$. 

Since $TM$ is locally Euclidean, we can find a relatively compact set $U_{(x,v)}$ containing $(x,v)$, such that $\overline{U_{(x,v)}} \subset U$.

Because the set $\{(x,v) \in TM| x \in K, v \in K_x\}$ is compact, we can find $(x_1,v_1),\ldots,(x_k,v_k)$ such that 
\[
\{(x,v) \in TM| x \in K, v\in K_x\} \subset \bigcup_{i=1}^k U_{(x_i,v_i)}.
\]
Consequently, we have that
\[
\sup_{x \in K} \sup_{v \in K_x} |\dd(\Exp_x)_v| \leq \max_{i=1}^k \sup_{(x,v) \in \overline{U_{(x_i,v_i)}}} |\dd(\Exp_x)_v|.
\]
It follows that we are done once we show that
\[
\sup_{(x,v) \in \overline{U_{(x_i,v_i)}}} |\dd(\Exp_x)_v| < \infty
\]
for all $i = 1,\ldots,k$. 

For this, remember that $F_{u,x_i,v_i}$ is continuous on $\overline{U_{(x_i,v_i)}}$, and hence bounded for any $u$, since $\overline{U_{(x_i,v_i)}}$ is compact. Consequently, it follows from the uniform boundedness principle that
\[
\sup_{(x,v) \in \overline{U_{(x_i,v_i)}}} |\tau_{\Exp_xv\Exp_{x_i}v_i}\dd(\Exp_x)_v\tau_{x_ix}|  < \infty
\]
However, because parallel transport is an isometry, we have 
\[
|\dd(\Exp_x)_v| = |\tau_{\Exp_xv\Exp_{x_i}v_i}\dd(\Exp_x)_v\tau_{x_ix}|,
\]
which concludes the proof. 
\end{proof}

As long as one restricts to a set where the inverse of the Riemannian exponential map is well-defined, one obtains in a similar way a bound for the differential of the inverse Riemannian exponential map. 

\begin{lemma}\label{lemma:uniform_bound_differential_inverse}
Let $K \subset M$ be compact and for any $x \in K$,  let $K_x \subset B(0,\iota(x)) \subset T_xM$ be compact. Assume that there exists a constant $C > 0$ such that $K_x \subset \overline{B(0,C)}$ for any $x \in K$. Then
\[
\sup_{x \in K} \sup_{v \in K_x} |\dd(\Exp_x)^{-1}_v| < \infty.
\]
\end{lemma}

\begin{remark}
When we take $K = \{x_0\}$ in Lemma \ref{lemma:uniform_bound_differential_inverse}, the statement simplifies as follows: If $\tilde K \subset B(0,\iota(x_0))$ is compact, then 
\[
\sup_{v \in \tilde K} |\dd(\Exp_{x_0})^{-1}_v| < \infty.
\]
\end{remark}

We are now in a position to find a first order Taylor expansion of the inverse Riemannian exponential map and control the error term appropriately. 

\begin{proposition}\label{proposition:inverse_exp}
Fix $x_0 \in M$ and let $K \subset B(0,\iota(x_0))$ be compact. Define $\tilde K = \Exp_{x_0}K$ and let $x \in \tilde K$ and $v \in T_xM$. Consider the geodesic $\gamma_v:[0,T] \to M$ defined by $\gamma_v(t) = \Exp_x(tv)$, where $T$ is such that the image of $\gamma_v$ is contained in $\tilde K$. Restrict $\Exp_{x_0}$ to $K$ and set $w(t) = \Exp_{x_0}^{-1}(\gamma_v(t)) \in K \subset T_{x_0}M$. Then there exists a constant $C > 0$ such that 
\[
|w(t) - w(0) - t\dd(\Exp_{x_0})^{-1}_{w(0)}(v)|_{g(x_0)} \leq Ct^2
\]
for all $t \in [0,T]$. Here, the constant $C$ only depends on the compact set $\tilde K$. 
\end{proposition}
\begin{proof}
First observe that $w(t)$ is well-defined, because $K \subset B(0,\iota(x_0))$ so that the restriction of $\Exp_{x_0}$ to $K$ is injective. Moreover, it is actually a diffeomorphism onto $\tilde K$, and thus $\dd(\Exp_{x_0})_w$ is also injective. By the inverse function theorem, $\Exp_{x_0}$ has a differentiable inverse, whose derivative at $w$ is given by
\[
\dd(\Exp_{x_0}^{-1})(w) = \dd(\Exp_{x_0})_w^{-1}
\]
Consequently, by Taylor's theorem, we find for any $t \in [0,T]$ that
\[
w(t) = w + t\dd(\Exp_{x_0})^{-1}_w(v) + t^2w''(\xi_t)
\]
for some $\xi_t \in (0,t)$.

To control the error term, we estimate $|w''(t)|$. We have
\[
w''(t) = \lim_{h\to0} \frac{\dd(\Exp_{x_0})^{-1}_{w(t+h)}(\dot\gamma_v(t+h)) - \dd(\Exp_{x_0})^{-1}_{w(t)}(\dot\gamma_v(t))}{h}.
\]
We estimate the numerator to find a desired bound on $w''(t)$. Set
\begin{equation}\label{eq:def_u}
u =  \dd(\Exp_{x_0})^{-1}_{w(t)}(\dot\gamma_v(t)) \in T_{x_0}M
\end{equation}
and
\begin{equation}\label{eq:def_u_tilde}
\tilde u = \dd(\Exp_{x_0})^{-1}_{w(t+h)}(\dot\gamma_v(t+h)) \in T_{x_0}M.
\end{equation}
Then 
\[
\dot\gamma_v(t) = \dd(\Exp_{x_0})_{w(t)}(u)
\]
and
\[
\dot\gamma_v(t+h) = \dd(\Exp_{x_0})_{w(t+h)}(\tilde u).
\]
As $\gamma_v$ is a geodesic, we have $\dot\gamma_v(t+h) = \tau_{\gamma_v(t)\gamma_v(t+h)}\dot\gamma_v(t)$. Consequently, we obtain
\begin{equation}\label{eq:exp}
\dd(\Exp_{x_0})_{w(t+h)}(\tilde u) = \tau_{\gamma_v(t)\gamma_v(t+h)}\dd(\Exp_{x_0})_{w(t)}(u).
\end{equation}

Define the curves $\psi_1, \psi_2$ in $T_{x_0}M$ by 
\[
\psi_1(s) = w(t) + su, \qquad \psi_2(s) = w(t+h) + s\tilde u
\]
and the corresponding curves $\phi_1,\phi_2$ in $M$ by
\[
\phi_1(s) = \Exp_{x_0}(w(t) + su), \qquad \phi_2(s) = \Exp_{x_0}(w(t+h) + s\tilde u).
\] 

The aim is to control $|u - \tilde u|_{g(x_0)}$.
For this, take normal coordinates around $x_0$ (which can be taken to cover all of $\tilde K$, because $\tilde K \subset \Exp_{x_0}[B(0,\iota(x_0))]$). In these coordinates, let us write $u = u^i\partial_i(x_0)$ and $\tilde u = \tilde u^j\partial_j(x_0)$. Note that in coordinates
\[
\phi_1(s) = (w^1(t) + su^1,\ldots,w^d(t) + su^d)
\]
and
\[
\phi_2(s) = (w^1(t+h) + s\tilde u^1,\ldots,w^d(t+h) + s\tilde u^d).
\]
Consequently,
\[
\dot\phi_1(s) = u^i\partial_i(\phi_1(s))
\]
and
\[
\dot\phi_2(s) = \tilde u^j\partial_j(\phi_2(s)).
\]
By equation \eqref{eq:exp} we have $\dot\phi_2(0) = \tau_{\gamma_v(t)\gamma_v(t+h)}\dot\phi_1(0)$. But then we find that the coefficients of $\dot\phi_2(0)$ satisfy the equations
\[
\dot V^k(s) + \Gamma_{ij}^k(\gamma_v(t + s))\dot\gamma_v^i(t + s)V^j(s) = 0
\]
with $V^k(0) = \dot\phi_1^k(0)$. Consequently, using a Taylor expansion, we find 
\[
\dot\phi_2^k(0) = \dot\phi_1^k(0) - h\Gamma_{ij}^k(\gamma_v(t))\dot\gamma_v^i(t)\dot\phi_1^j(0) + \Oo(h^2)
\]
Using that $\dot\phi_1^k(0) = u^k$ and $\dot\phi_2^k(0) = \tilde u^k$, we obtain
\begin{equation}\label{eq:difference_diff_exp}
u^k - \tilde u^k = h\Gamma_{ij}^k(\gamma_v(t))\dot\gamma_v^i(t)u^j + \Oo(h^2).
\end{equation}

Because we are using normal coordinates around $x_0$, we have
\[
|u - \tilde u|_{g(x_0)}^2 = \sum_{k=1}^d (u^k - \tilde u^k)^2.
\]
If we plug in expression \eqref{eq:difference_diff_exp}, we get
\[
|u - \tilde u|_{g(x_0)}^2 = h^2\sum_{k=1}^d (\Gamma_{ij}^k(\gamma_v(t))\dot\gamma_v^i(t)u^j)^2 + \Oo(h^3).
\]
As the Christoffel symbols are continuous, they are bounded on our compact set $\tilde K$ by some constant $C_1$. Furthermore, the coefficients $g_{ij}$ of the metric are also continuous, and in particular, by the positive definiteness of the metric, there exists a uniform constant $\delta > 0$ such that $g_{ii}(x) \geq \delta$ for all $x \in \tilde K$ and all $i = 1,\ldots,d$.  In particular, this implies
\[
(\dot\gamma_v^i(t))^2 \leq \delta^{-1}|\dot\gamma_v(t)|_{g(\gamma_v(t))}^2 = \delta^{-1}|v|_{g(\gamma_v(0))}^2.
\]
Similarly, we have
\[
(u^j)^2 \leq |\dd(\Exp_{x_0})^{-1}_{w(t)}(\dot\gamma_v(t))|^2_{g(x_0)} \leq C_2^2|\dot\gamma_v(t)|^2_{g(\gamma_v(t))} = C_2^2|v|^2_{g(\gamma_v(0))},
\]
where we used Lemma \ref{lemma:uniform_bound_differential_inverse} to find the constant $C_2$, which again only depends on the compact set $\tilde K$. 

Collecting everything, we find
\[
|u - \tilde u|_{g(x_0)}^2 \leq C_1^2C_2^2|v|^2_{g(\gamma_v(t))}h^2 + \Oo(h^3)
\]
Recalling the definition of $u$ and $\tilde u$ in \eqref{eq:def_u} and \eqref{eq:def_u_tilde} respectively, we find after taking the limit $h \to 0$ that
\[
|w''(t)|_{g(x_0)} \leq C_1C_2|v|_{g(\gamma_v(0))},
\]
which provides the desired constant, because $C_1,C_2$ only depend on $\tilde K$. 
\end{proof}

\subsection{Differential of the Riemannian exponential map and parallel transport} \label{section:Taylor_differential_exponential}

Next, we wish to understand the relation between the differential of the Riemannian exponential map and parallel transport. Before we can make the appropriate comparison, we first need a version of Taylor's theorem suitable for vector fields along a curve on a manifold. 

\begin{proposition}[Taylor's theorem]\label{proposition:Taylor}
Let $\gamma$ be a curve in $M$ and $v$ a vector field along $\gamma$. Define $D_tv(t) := \Nabla_{\dot\gamma(t)}v(t)$ and $D_t^k$ as the $k$-th covariant derivative in this way. Fix $n \in \NN$. For every $t > 0$ there exists $\xi_t \in (0,t)$ such that
\[
v(t) = \sum_{k=0}^n \frac{t^k}{k!}\tau_{\gamma(0)\gamma(t)}D_t^kv(0) + \frac{t^{k+1}}{(k+1)!}\tau_{\gamma(\xi_t)\gamma(t)}D_t^{k+1}v(\xi_t).
\]
\end{proposition}
\begin{proof}
Consider the map $f(t) = \tau_{\gamma(0)\gamma(t)}^{-1}v(t)$, mapping into $T_{\gamma(0)}M$. Because $f$ is smooth, by Taylor's theorem, given $n \in \NN$ and $t > 0$, there exists $\xi_t \in (0,t)$ such that
\[
f(t) = \sum_{k=0}^n \frac{t^k}{k!}f^{(k)}(0) + \frac{t^{k+1}}{(k+1)!}f^{(k+1)}(\xi_t).
\]
Let us compute the derivatives of $f$. Note that 
\begin{align*}
f'(t)
&= 
\lim_{h\to 0}\frac{f(t+h) - f(t)}{h}
\\
&=
\lim_{h\to0} \frac{\tau_{\gamma(0)\gamma(t+h)}^{-1}v(t+h) - \tau_{\gamma(0)\gamma(t)}^{-1}v(t)}{h}
\\
&=
\tau_{\gamma(0)\gamma(t)}^{-1}\lim_{h\to0} \frac{\tau_{\gamma(t)\gamma(t+h)}^{-1}v(t+h) - v(t)}{h}
\\
&=
\tau_{\gamma(0)\gamma(t)}^{-1}D_tv(t).
\end{align*}
Using induction, one can show that
\[
f^{(k)}(t) = \tau_{\gamma(0)\gamma(t)}^{-1} D_t^kv(t)
\]
for all $k \in \NN$. But then we find that
\[
\tau_{\gamma(0)\gamma(t)}^{-1}v(t) =  \sum_{k=0}^n \frac{t^k}{k!} D_t^kv(0) + \frac{t^{k+1}}{(k+1)!}\tau_{\gamma(0)\gamma(\xi_t)}^{-1}D_t^{k+1}v(\xi_t).
\]
Applying $\tau_{\gamma(0)\gamma(t)}$ to both sides and observing that $t > \xi_t$ gives the desired result.
\end{proof}

We are now able to compare the differential of the Riemannian exponential map and parallel transport. The Taylor series of the differential of the exponential map may be found in e.g \cite[Appendix A]{Wal12}. The error term for finite Taylor polynomials seems to belong to folklore, but we insert a proof here for the reader's convenience. 

\begin{proposition}\label{proposition:difexp_parallel}
Let $x_0 \in M$ and take $w,u \in T_{x_0}M$. Consider the geodesic $\gamma_w:[0,1] \to M$ given by $\gamma_w(t) = \Exp_{x_0}(tw)$. For every $t \in [0,1]$ there exists $\xi_t \in (0,t)$ such that
\[
\dd(\Exp_{x_0})_{tw}(u) = \tau_{\gamma_w(0)\gamma_w(t)}u + \frac12t\tau_{\gamma_w(\xi_t)\gamma_w(t)}R_{\gamma_w(\xi_t)}(\dd(\Exp_{x_0})_{\xi_tw}(\xi_tu), \dot\gamma_w(\xi_t))\dot\gamma_w(\xi_t).
\]
\end{proposition}

\begin{proof}
Consider the vector field $J(t) = \dd(\Exp_{x_0})_{tw}(tu)$ along $\gamma_w(t)$. As argued in Section \ref{section:Jacobi}, $J(t)$ is a Jacobi field along $\gamma(t)$ with $J(0) = 0$ and $\dot J(0) = u$. By the Jacobi equation \eqref{eq:jacobi}, the second derivative is given by
\[
D_t^2 J(t) = - R_{\gamma_w(t)}(J(t),\dot\gamma_w(t))\dot\gamma_w(t).
\]
Consequently, by Proposition \ref{proposition:Taylor} we find there exists some $\xi_t \in (0,t)$ such that
\[
J(t) = t\tau{\gamma_w(0)\gamma_w(t)}u - \frac12t^2\tau_{\gamma_w(\xi_t)\gamma_w(t)}R_{\gamma_w(\xi_t)}(\dd(\Exp_{x_0})_{\xi_tw}(\xi_tu), \dot\gamma_w(\xi_t))\dot\gamma_w(\xi_t).
\]
The result now follows after dividing by $t$.
\end{proof}

This proposition allows us to obtain the following estimate.

\begin{corollary}\label{cor:exp_parallel}
Fix $x_0 \in M$ and let $w \in B(0,\iota(x_0))\subset T_{x_0}M$. Define the geodesic $\gamma_w:[0,1] \to M$ by $\gamma_w(t) = \Exp_{x_0}(tw)$. There exists a constant $C > 0$ only depending on some compact set containing $\gamma_w$ such that
\[
|\dd(\Exp_{x_0})_w(u) - \tau_{\gamma_w(0)\gamma_w(1)}u|_{g(\gamma_w(1))} \leq C|u|_{g(x_0)}|w|_{g(x_0)}^2
\]
for all $u \in T_{x_0}M$.
\end{corollary}
\begin{proof}
By Proposition \ref{proposition:difexp_parallel} there exists $\xi \in (0,1)$ such that
\[
\dd(\Exp_{x_0})_w(u) - \tau_{\gamma_w(0)\gamma_w(1)}u = - \frac12\tau_{\gamma_w(\xi),\gamma_w(1)}R_{\gamma_w(\xi)}(\dd(\Exp_{x_0})_{\xi w}(\xi u), \dot\gamma_w(\xi))\dot\gamma_w(\xi).
\]
Now taking norms on both sides, we first observe that the norm of the Riemann curvature endomorphism is bounded on compact sets, because it is continuous (in coordinates the norm can be expressed as a continuous functions of the coefficients). Furthermore, by Lemma \ref{lemma:uniform_bound_differential} we have that $w \mapsto |\dd(\Exp_{x_0})_w|$ is bounded on compact sets.

We thus obtain constants $C_1,C_2 > 0$, only depending on some compact set containing the curve $\gamma_w$ such that
\begin{align*}
&|\dd(\Exp_{x_0})_w(u) - \tau_{\gamma_w(0)\gamma_w(1)}u|_{g(\gamma_w(1))}
\\
&\leq
\frac12|R_{\gamma_w(\xi)}(\dd(\Exp_{x_0})_{\xi w}(\xi u), \dot\gamma_w(\xi))\dot\gamma_w(\xi)|_{g(\gamma_w(\xi))}
\\
&\leq
C_1|\dd(\Exp_{x_0})_{\xi w}(\xi u)|_{g(\gamma_w(\xi))}|\dot\gamma_w(\xi)|_{g(\gamma_w(\xi))}^2
\\
&\leq
C_1C_2|u|_{g(x_0)}|w|^2_{g(x_0)}.
\end{align*}
Here, in the last line we used that $\xi < 1$ and the fact that $\gamma_w$ is a geodesic.
\end{proof}

The result in the latter corollary can also be used to compare the inverse of the differential of the exponential map to the inverse of parallel transport, which itself is parallel transport, but in the opposite direction.

\begin{corollary}\label{corollary:inverse_exp_parallel}
Let $x_0 \in M$ and fix $w \in B(0,\iota(x_0)) \subset T_{x_0}M$. Define the geodesic $\gamma_w:[0,1] \to M$ by $\gamma_w(t) = \Exp_{x_0}(tw)$. Then there exists a constant $C > 0$ only depending on some compact set containing $\gamma_w$, such that
\[
|\dd(\Exp_{x_0})^{-1}_w(u) - \tau^{-1}_{\gamma_w(0)\gamma_w(1)}u|_{g(\gamma_w(1))} \leq C|u|_{g(\gamma_w(1))}|w|_{g(x_0)}^2
\]
for all $u \in T_{\gamma_w(1)}M$.
\end{corollary}
\begin{proof}
Fix $u \in T_{\gamma_w(1)}M$ and consider $\dd(\Exp_{x_0})_w^{-1}u \in T_{x_0}M$. By Corollary \ref{cor:exp_parallel}, there exists a constant $C > 0$ only depending on a compact set containing $\gamma_w$ such that
\[
|u - \tau_{\gamma_w(0)\gamma_w(1)}\dd(\Exp_{x_0})_w^{-1}u|_{g(\gamma_w(1))} \leq C|\dd(\Exp_{x_0})_w^{-1}u|_{g(x_0)}|w|_{g(x_0)}^2.
\]
Because parallel transport is an isometry, the left hand side is equal to
\[
|\tau_{\gamma_w(1)\gamma_w(0)}u - \dd(\Exp_{x_0})_w^{-1}u|_{g(\gamma_w(1))}.
\]
For the right hand side, we observe that by Lemma \ref{lemma:uniform_bound_differential_inverse} there exists a constant $\tilde C > 0$, only depending on some compact set containing $\gamma_w$ such that
\[
|\dd(\Exp_{x_0})_w^{-1}u|_{g(x_0)} \leq \tilde C|u|_{g(\gamma_w(1))}.
\]
Putting everything together, we find 
\[
|\tau_{\gamma_w(1)\gamma_w(0)}u - \dd(\Exp_{x_0})_w^{-1}u|_{g(\gamma_w(1))} \leq C\tilde C|u|_{g(\gamma_w(1))}|w|_{g(x_0)}^2
\]
as desired.
\end{proof}

\subsection{Spreading of geodesics}

We conclude this section with a result on how far geodesics, possibly starting in different points, can spread in a given amount of time. To shed some light on the upcoming result, we first consider the Euclidean case. For this, let $\gamma(t) = \gamma(0) + t\dot\gamma(0)$ and $\phi(t) = \phi(0) + t\dot\phi(t)$ be two straight lines. Then
\[
|\gamma(t) - \phi(t)|^2 = |\gamma(0) - \phi(0)|^2 + 2t\inp{\dot\gamma(0) - \dot\phi(0)}{\gamma(0) - \phi(0)} + t^2|\dot\gamma(t) - \dot\phi(t)|^2.
\]
It turns out that in a Riemannian manifold, this formula is analogous up to first order. The curvature terms show up in the second order term. Before we prove this, we first need a lemma.

\begin{lemma}\label{lemma:continuity_distance}
Let $K \subset M$ be compact and fix $L > 0$. Let $0 < r < \iota(K)$. Let $\phi:[0,T] \to M$ and $\gamma:[0,T] \to M$ be two geodesics contained in $K$. Assume that $d(\phi(0),\gamma(0)) \leq \frac r2$ and $|\dot\phi(0)|,|\dot\gamma(0)| \leq L$. Then there exists a $t_0 > 0$, only depending on $K,L$ and $r$, such that for all $0 \leq t \leq t_0$ we have
\[
d(\phi(t),\gamma(t)) < r.
\]
\end{lemma}
\begin{proof}
Because $d:M\times M \to \RR$ is continuous, and $K \times K$ is compact, $d(\cdot,\cdot)$ is uniformly continuous on $K \times K$. Consequently, pick $\epsilon > 0$ such that $|d(x,y) - d(x',y')| < \frac r2$, whenever $d(x,x') < \epsilon$ and $d(y,y') < \epsilon$.

Now observe that $d(\phi(t),\phi(0)) \leq t|\dot\phi(0)| \leq tL$ and likewise $d(\gamma(t),\gamma(0)) \leq tL$. Hence, if we take $t_0 < \epsilon L^{-1}$, then for all $0 \leq t \leq t_0$ we have $d(\phi(t),\phi(0)) < \epsilon$ and $d(\gamma(t),\gamma(0)) < \epsilon$. By the choice of $\epsilon$, it follows that
\[
|d(\phi(0),\gamma(0)) - d(\phi(t),\gamma(t))| < \frac r2.
\]
Since $d(\phi(0),\gamma(0)) \leq \frac12r$, the above then implies that $d(\phi(t),\gamma(t)) < r$ as desired.
\end{proof}

\begin{proposition}\label{proposition:distance_geodesic_point}
Let $K \subset M$ be compact and fix $L > 0$. Let $0 < r < \iota(K)$ and fix $t_0 > 0$ as in Lemma \ref{lemma:continuity_distance}. Let $\phi:[0,t_0] \to M$ and $\gamma:[0,t_0] \to M$ be two geodesics in $K$ such that $d(\gamma(0),\phi(0)) \leq \frac r2$ and $|\dot\phi(0)|,|\dot\gamma(0)| \leq L$. Finally, let $\tilde K$ be a compact set containing all geodesics of minimal length between points in $K$. Then for all $0 \leq t \leq t_0$ we have
\begin{align*}
&d(\gamma(t),\phi(t))^2 
\\
&\leq 
d(\gamma(0),\phi(0))^2 + 2t\inp{\tau_{\phi(0)\gamma(0)}^{-1}\dot\gamma(0) - \dot\phi(0)}{\Exp_{\phi(0)}^{-1}\gamma(0)} + t^2C(|\dot\gamma(0)| + |\dot\phi(0)|),
\end{align*}
where the constant $C > 0$ only depends on $\tilde K,L$ and $r$.
\end{proposition}
\begin{proof}
Define $f(t) = d(\gamma(t),\phi(t))^2$. By the choice of $t_0$, Lemma \ref{lemma:continuity_distance} gives us that 
\[
d(\phi(t),\gamma(t)) < r < \iota(K)
\]
for every $0 \leq t \leq t_0$. Consequently, $\phi(t)$ and $\gamma(t)$ may be joined by a unique geodesic of minimal length. Moreover, by restricting $\Exp$, we have $f(t) = |\Exp_{\phi(t)}^{-1}\gamma(t)|^2$. Consequently, we can compute
\begin{align*}
f'(t) 
&= 
\frac{\dd}{\dd t}|\Exp_{\phi(t)}^{-1}\gamma(t)|^2
\\
&=
2\inp{\Nabla_{\dot\phi(t)} \Exp_{\phi(t)}^{-1}\gamma(t)}{\Exp_{x_0}^{-1}\gamma(t)}. 
\end{align*}

Now define the variation of curves $\Gamma:[0,t_0]\times[0,1] \to M$ by 
\[
\Gamma(t,s) = \Exp_{\phi(t)}(s\Exp_{\phi(t)}^{-1}\gamma(t)).
\]
Then for each $t$, the curve $s \mapsto \Gamma(t,s)$ is the geodesic of minimal length between $\phi(t)$ and $\gamma(t)$. Hence, $\Gamma([0,t_0]\times[0,1]) \subset \tilde K$. Furthermore, because $\Gamma$ is a variation of geodesics, the vector field
\[
J_t(s) = \partial_t\Gamma(t,s)
\]
is a Jacobi field along the curve $\Gamma_t(s) := \Gamma(t,s)$ for all $0 \leq t \leq t_0$.

Now note that by the Symmetry Lemma (Lemma \ref{lemma:symmetry}), we have
\[
\Nabla_{\dot\phi(t)} \Exp_{\phi(t)}^{-1}\gamma(t) = D_t\partial_s\Gamma(t,0) = D_s\partial_t\Gamma(t,0) = \dot J_t(0).
\]
Consequently, we obtain
\[
f'(t) = 2\inp{\dot J_t(0)}{\Exp_{x_0}^{-1}\gamma(t)} = 2\inp{\dot J_t(0)}{\partial_s\Gamma(t,0)}.
\]

By Proposition \ref{prop:inner_product_jacobi} we find
\begin{align*}
f'(t)
&=
2\inp{\dot J_t(0)}{\partial_s\Gamma(t,0)}
\\
&=
2\inp{J_t(1)}{\partial_s\Gamma(t,1)} - 2\inp{J_t(0)}{\partial_s\Gamma(t,0)} 
\\
&=
2\inp{\dot\gamma(t)}{-\Exp_{\gamma(t)}\phi(t)} - 2\inp{\dot\phi(t)}{\Exp_{\phi(t)}^{-1}\gamma(t)}
\\
&= 
2\inp{\tau_{\phi(t)\gamma(t)}^{-1}\dot\gamma(t) - \dot\phi(t)}{\Exp_{\phi(t)}^{-1}\gamma(t)}.
\end{align*}
Consequently, we have
\[
f'(0) = 2\inp{\tau_{\phi(0)\gamma(0)}^{-1}\dot\gamma(0) - \dot\phi(0)}{\Exp_{\phi(0)}^{-1}\gamma(0)}.
\]
By Taylor's theorem, we find that
\[
d(\gamma(t),\phi(t))^2 \leq d(\gamma(0),\phi(0))^2 + 2t\inp{\tau_{\phi(0)\gamma(0)}^{-1}\dot\gamma(0) - \dot\phi(0)}{\Exp_{\phi(0)}^{-1}\gamma(0)} + \frac12t^2\sup_{\xi \in [0,t]} |f''(\xi)|.
\]

We now turn to estimating the residual term. For this, we compute $f''(t)$ as follows:
\begin{align*}
\frac12f''(t)
&=
\frac{\dd}{\dd t}\inp{\dot\gamma(t)}{-\Exp_{\gamma(t)}\phi(t)} - \frac{\dd}{\dd t}\inp{\dot\phi(t)}{\Exp_{\phi(t)}^{-1}\gamma(t)}
\\
&=
-\inp{\dot\gamma(t)}{\Nabla_{\dot\gamma(t)}\Exp_{\gamma(t)}^{-1}\phi(t)} - \inp{\dot\phi(t)}{\Nabla_{\dot\phi(t)}\Exp_{\phi(t)}^{-1}\gamma(t)}
\\
&=
\inp{\dot\gamma(t)}{\partial_t\Gamma(t,1)} - \inp{\dot\phi(t)}{\partial_t\Gamma(t,0)}
\\
&=
\inp{\dot\gamma(t)}{\dot J_t(1)} - \inp{\dot\phi(t)}{\dot J_t(0)} .
\end{align*}
Here we used that $\Nabla_{\dot\phi(t)}\dot\phi(t) = \Nabla_{\dot\gamma(t)}\dot\gamma(t) = 0$, since $\phi$ and $\gamma$ are geodesics. Consequently, we have
\[
\frac12|f''(t)| \leq |\dot\gamma(t)||\dot J_t(1)| + |\dot\phi(t)||\dot J_t(0)| = |\dot\gamma(0)||\dot J_t(1)| + |\dot\phi(0)||\dot J_t(0)|,
\]
where we again used that $\gamma$ and $\phi$ are geodesics. It follows that we are done once we can bound $|\dot J_t(0)|$ and $|\dot J_t(1)|$. For this, we first obtain a more specific expression for the Jacobi field $J_t$. To this end, we define for every $0 \leq t \leq t_0$ the vector fields
\[
J_t^1(s) = \dd(\Exp_{\phi(t)})_{s\partial_s\Gamma(t,0)}(s\dot J^1_t(0))
\]
and 
\[
J_t^2(s) = \dd(\Exp_{\gamma(t)})_{-s\partial_s\Gamma(t,1)}(s\dot J^2_t(0)),
\]
where
\[
\dot J^1_t(0) = \dd(\Exp_{\phi(t)})^{-1}_{\Exp_{\phi(t)}^{-1}\gamma(t)}\dot\gamma(t) \in T_{\phi(t)}M
\]
and likewise
\[
\dot J^2_t(0) = \dd(\Exp_{\gamma(t)})^{-1}_{\Exp_{\gamma(t)}^{-1}\phi(t)}\dot\phi(t) \in T_{\gamma(t)}M.
\]
As explained in Section \ref{section:Jacobi}, $J_t^1$ and $J_t^2$ are Jacobi fields along $\Gamma_t$. Note that $J_t^1(0) = J_t^2(0) = 0$ and $J_t^1(1) = \dot\gamma(t)$ and $J_t^2(1) = \dot\phi(t)$. Because $J_t$ is the unique Jacobi field along $\Gamma_t$ with $J_t(0) = \dot\phi(t)$ and $J_t(1) = \dot\gamma(t)$, it follows that
\[
J_t(s) = J_t^1(s) + J_t^2(1-s).
\]

Using the above decomposition, we show how to bound $|\dot J_t(0)|$. The bound for $|\dot J_t(1)|$ may be obtained similarly. By the triangle inequality, we have
\[
|\dot J_t(0)| \leq |\dot J^1_t(0)| + |\dot J^2_t(1)|.
\]

Note that
\[
|\dot J_t^1(0)| = |\dd(\Exp_{\phi(t)})^{-1}_{\Exp_{\phi(t)}^{-1}\gamma(t)}\dot\gamma(t)| \leq |\dd(\Exp_{\phi(t)})^{-1}_{\Exp_{\phi(t)}^{-1}\gamma(t)}||\dot\gamma(t)| 
\]
Consequently, by Lemma \ref{lemma:uniform_bound_differential_inverse} there exists a constant $C > 0$ only depending on $K$ and $r$ (since $|\Exp_{\phi(t)}^{-1}\gamma(t)| = d(\phi(t),\gamma(t)) \leq r$) such that
\[
|\dot J_t^1(0)| \leq C|\dot\gamma(t)| = C|\dot\gamma(0)|.
\]

For the other term, it follows from Proposition \ref{prop:norm_jacobi} that 
\begin{align*}
|\dot J_t^2(1)|
&\leq
|\dot J_t^2(0)| + \sup_{s \in [0,1]} |R_{\Gamma(t,s)}(J_t^2(s),\partial_s\Gamma(t,s))\partial_s\Gamma(t,s)|
\\
&\leq 
C|\dot\phi(0)| + |\partial_s\Gamma(t,0)|^2 \sup_{s \in [0,1]} |R_{\psi_t(s)}||J_t^2(s)|
\\
&\leq
C|\dot\phi(0)| + \tilde Cd(\gamma(t),\phi(t))^2\sup_{s \in [0,1]} |J_t^2(s)|
\\
&\leq
C|\dot\phi(0)| + \tilde Cr^2\sup_{s \in [0,1]} |J_t^2(s)|.
\end{align*}
Here we used in the second line again Lemma \ref{lemma:uniform_bound_differential_inverse} as above, together with the fact that the curves $\Gamma_t(s)$ are geodesics. Furthemore, we used that the curvature is continuous, and hence bounded on compact sets, so that $\tilde C$ only depends on $\tilde K$, since the variation $\Gamma$ is contained in $\tilde K$. In the last line, we used that $d(\gamma(t),\phi(t)) \leq r$ for all $0 \leq t \leq t_0$ by choice of $t_0$.

Finally, we have for any $s \in [0,1]$
\begin{align*}
|J_t^2(s)|
&=
|\dd(\Exp_{\gamma(t)})_{-s\partial_s\Gamma(t,1)}(s\dot J^2_t(0))|
\\
&\leq
s|\dd(\Exp_{\gamma(t)})_{-s\partial_s\Gamma(t,1)}||\dot J^2_t(0))|
\\
&\leq
C'|\dot\phi(0)|,
\end{align*}
where in the last line we used Lemma \ref{lemma:uniform_bound_differential}. Collecting everything, there exists a constant $C > 0$, only depending on $\tilde K$ and $r$, such that
\[
|\dot J_t^2(1)| \leq C|\dot\phi(0)|.
\] 

Putting everything together, we find that
\[
|\dot J_t(0)| \leq |\dot J^1_t(0)| + |\dot J^2_t(1)| \leq C(|\dot\gamma(0)| + |\dot\phi(0)|)
\]
for some $C > 0$ only depending on $\tilde K$ and $r$. Obtaining a similar bound for $|\dot J_t(1)|$ now proves the claim. 
\end{proof}


\section{Proof of Cram\'er's theorem for geodesic random walks}\label{section:Cramer_bounded}

In this section we provide a proof of Cram\'er's theorem for geodesic random walks with independent and identically distributed increments, which are bounded and have expectation 0. The proof relies on an analysis of the geometric properties of a geodesic random walk. To prove the theorem, we follow the steps as discussed in Section \ref{section:approach_proof}. We provide the details and show how we use the geometric results from Section \ref{section:geometry}. For completeness, let us recall the statement of the theorem.

\begin{theorem}[Cram\'er's theorem for Riemannian manifolds]\label{theorem:Cramer_sect_proof}
Let $(M,g)$ be a complete Riemannian manifold. Fix $x_0 \in M$ and let $\{\mu_x\}_{x\in M}$ be a collection of measures such that $\mu_x \in \cP(T_xM)$ for all $x \in M$.  For every $n \geq 1$, let $\{(\frac1n*\Ss)_j\}_{j\geq0}$ be a $\frac1n$-rescaled geodesic random walk started at $x_0$ with independent increments $\{X_j^n\}_{j\geq1}$, compatible with $\{\mu_x\}_{x\in M}$. Let $\{(\frac1n*\Ss)_n\}_{n \geq 0}$ be the associated empirical average process started at $x_0$.  Assume the increments are bounded and have expectation 0. Assume furthermore that the collection $\{\mu_x\}_{x\in M}$ satisfies the consistency property in Definition \ref{definition:consistency}. Then $\{(\frac1n*\Ss)_n\}_{n\geq0}$ satisfies in $M$ the LDP with good rate function
\begin{equation}\label{eq:rate_cramer_sect_proof}
I_M(x) = \inf\{\Lambda_{x_0}^*(v)| v \in \Exp_{x_0}^{-1}x\}
\end{equation}

\end{theorem}

In Section \ref{section:upper_bound} we prove the upper bound of the large deviation principle for $\{(\frac1n*\Ss)_n\}_{n\geq 1}$ in $M$, while in Section \ref{section:lower_bound} we prove the lower bound. More specifically, Theorem \ref{theorem:Cramer_sect_proof} follows immediately from Proposition \ref{prop:upper_bound_LDP} together with Proposition \ref{prop:lower_bound_LDP}.\\

However, before we can prove the upper and lower bound of the large deviation principle for $\{(\frac1n*\Ss)_n\}_{n\geq 1}$, we first need some general results and estimates. From here on, we fix $r > 0$ to be the uniform bound on the increments of the random walk. By the triangle inequality, we find
\[
d\left(\left(\frac1n*\Ss\right)_k,x_0\right) \leq \frac1n\sum_{l=1}^k |X_k^n| \leq \frac knr \leq r
\] 
for all $0 \leq k \leq n$. Consequently, for every $n \geq 0$ and $1 \leq k \leq n$ we have
\[
\left(\frac1n*\Ss\right)_k \in \overline{B(x_0,r)} =: K.
\]
By completeness of $M$, $K$ is compact since it is closed and bounded.

Now consider the process $Z_n$ in $T_{x_0}M$ given by
\[
Z_n = \frac1n\sum_{k=1}^n \tau_{x_0(\frac1n*\Ss)_{k-1}}^{-1}X_k^n.
\]
Here, the parallel transport $\tau_{x_0(\frac1n*\Ss)_{k-1}}$ is considered along the piecewise geodesic path traced out by the geodesic random walk. From Cram\'er's theorem for vector spaces it follows that $\{Z_n\}_{n\geq0}$ satisfies the large deviation principle in $T_{x_0}M$, which we will show in the following proposition.

\begin{proposition}\label{prop:LDP_Exp}
Let the assumptions of Theorem \ref{theorem:Cramer_sect_proof} be satisfied. For every $n \geq 0$, define $Z_n = \frac1n\sum_{k=1}^n \tau_{x_0(\frac1n*\Ss)_{k-1}}^{-1}X_k^n \in T_{x_0}M$. Let $\Lambda_{x_0}(\lambda) = \log\EE(e^{{\lambda}{X_1}})$ be the log moment generating function of the increments. Then $\{Z_n\}_{n\geq0}$ satisfies the large deviation principle in $T_{x_0}M$ with good rate function 
\[
I(v) = \Lambda_{x_0}^*(v) := \sup_{\lambda \in T_{x_0}M} \{\inp{\lambda}{v} - \Lambda_{x_0}(\lambda)\}.
\]
\end{proposition}
\begin{proof}
Define $Y_k^n = \tau_{x_0(\frac1n*\Ss)_{k-1}}^{-1}X_k^n \in T_{x_0}M$. We compute for any $\lambda \in T_{x_0}M$
\begin{align*}
\EE(e^{\inp{\lambda}{Y_k^n}})
&=
\EE\left(\EE\left(e^{\inp{\lambda}{\tau_{x_0(\frac1n*\Ss)_{k-1}}^{-1}X_k^n}}|\Ff_{k-1}\right)\right)
\\
&=
\EE\left(\int_{T_{(\frac1n*\Ss)_{k-1}}M} e^{\inp{\lambda}{\tau_{x_0(\frac1n*\Ss)_{k-1}}^{-1}v}} \mu_{(\frac1n*\Ss)_{k-1}}(\dd v)\right)
\\
&=
\EE\left(\int_{T_{x_0}M} e^{\inp{\lambda}{v}} \mu_{x_0}(\dd v)\right)
\\
&=
\int_{T_{x_0}M} e^{\inp{\lambda}{v}} \mu_{x_0}(\dd v).
\end{align*}
Here we used in the second line that $\tau_{x_0(\frac1n*\Ss)_{k-1}}^{-1}$ is measurable with respect to $\Ff_{k-1}$, together with the fact that the increments are independent (see Definition \ref{definition:independent_increments}). In the third line we applied Proposition \ref{prop:logmgf}, using that the increments are identically distributed. It follows that $Y_k^n$ is distributed according to $\mu_{x_0}$.  

Consequently, the result follows from Cram\'er's theorem once we show that $Y_k^n$ and $Y_l^n$ are independent whenever $k \neq l$. To this end, assume without loss of generality that $l < k$. Then for measurable sets $A,B \subset T_{x_0}M$ we find in a similar way as above that
\begin{align*}
&\PP(Y_l^n \in A, Y_k^n \in B)
\\
&=
\EE(I(Y_l^n \in A)\EE(I(Y_k^n \in B) | \Ff_{k-1}))
\\
&=
\EE\left( I(Y_l^n \in A) \int_{T_{(\frac1n*\Ss)_{k-1}}M} I\left(\tau_{x_0(\frac1n*\Ss)_{k-1}}^{-1}v \in B\right) \mu_{(\frac1n*\Ss)_{k-1}}(\dd v)\right)
\\
&=
\EE\left( I(Y_l^n \in A) \int_{T_{x_0}M} I\left(v \in B\right) \mu_{x_0}(\dd v)\right)
\\
&=
\EE(I(Y_l^n \in A))\EE(I(Y_k^n \in B))
\\
&=
\PP(Y_l^n \in A)\PP(Y_k^n \in B),
\end{align*}
where $I$ denotes the indicator function. Above, we used in the one but last line that $Y_k^n$ is distributed according to $\mu_{x_0}$. We conclude that the $Y_l^n$ and $Y_k^n$ are independent. 
\end{proof}

\begin{remark}
Note that in the proof of Proposition \ref{prop:LDP_Exp} we did not use along which path we performed the parallel transport $\tau_{x_0(\frac1n*\Ss)_{k-1}}^{-1}$, only that it was measurable with respect to $\Ff_{k-1}$. Consequently, the result holds for any choice of parallel transport, as long as it is measurable with respect to $\Ff_{k-1}$. 
\end{remark}

Proposition \ref{prop:LDP_Exp} suggests we should try to map the sequence $\{(\frac1n*\Ss)_n\}_{n\geq0}$ from $M$ to $T_{x_0}M$ in such a way that it will be close to the sequence $\{Z_n\}_{n\geq0}$.

To this end, recall that if we assume that  $r < \iota(x_0)$, then for all $n$ and all $0 \leq k \leq n$ we can uniquely define 
\[
v_k^n \in \Exp_{x_0}^{-1}\left(\left(\frac1n*\Ss\right)_k\right) \subset T_{x_0}M
\]
with $|v_k^n| < \iota(x_0)$, because $d((\frac1n*\Ss)_k,x_0) \leq r < \iota(x_0)$. 

As explained in Step 2 of Section \ref{section:sketch_upper_bound}, we have the following estimate. The first term of the upper bound in \eqref{eq:sum_parallel_transport} follows from replacing $v_n^l$ with a sum of differentials of the Riemannian exponential map, while the second term follows from replacing these differentials with parallel transport.

\begin{proposition}\label{prop:comparison_independent_sum}
Let the assumptions of Theorem \ref{theorem:Cramer_sect_proof} be satisfied. Additionally, let $r$ be the uniform bound of the increments and assume that $r < \iota(x_0)$. Then there exists a constant $C > 0$ such that
\begin{equation}\label{eq:sum_parallel_transport}
\left|v_l^n - \frac1n\sum_{k=1}^l \tau_{x_0(\frac1n*\Ss)_{k-1}}^{-1}X_k^n\right| \leq C\frac l{n^2} + Cr^2\frac{l^3}{n^3}
\end{equation}
for all $n$ and all $1 \leq l \leq n$. 
\end{proposition}
\begin{proof}
Recall that for all $n$ and all $0 \leq k \leq n$ we have that $(\frac1n*\Ss)_k$ is in the compact set $K = \overline{B(x_0,r)}$. Consequently, 
\[
v_k^n \in \overline{B(0,r)} \subset T_{x_0}M
\] 
for all $n$ and all $0 \leq k \leq n$. But then it follows from Proposition \ref{proposition:inverse_exp} that for every $0 \leq k \leq n$ there exists a constant $C_k > 0$ only depending on the norms of $v^n_k,v^n_{k+1}$ and $X_k^n$ such that
\begin{equation}\label{eq:induction_step}
\left|v_{k+1}^n - \left(v_k^n + \frac1n\dd(\Exp_{x_0})^{-1}_{v_k^n}X_{k+1}^n\right)\right| \leq  C_k\frac1{n^2}.
\end{equation}

Because each of the norms $|v_k^n|, |v_{k+1}^n|$ and $|X_k^n|$ are bounded by $r$, we conclude that we can take $C_k = C$ independent of $k$.\\

Turning to the proof of the statement, by the triangle inequality we have
\begin{align*}
&\left|v_l^n - \frac1n\sum_{k=1}^l \tau_{x_0(\frac1n*\Ss)_{k-1}}^{-1}X_k^n\right| 
\\
&\leq
\left|v_l^n - \frac1n\sum_{k=1}^l \dd(\Exp_{x_0})_{v_{k-1}^n}^{-1}X_k^n\right| + \frac1n\sum_{k=1}^l\left|\dd(\Exp_{x_0})_{v_{k-1}^n}^{-1}X_k^n -  \tau_{x_0(\frac1n*\Ss)_{k-1}}^{-1}X_k^n\right|.
\end{align*}
We estimate both terms separately.

For the first term, we write $v_l^n$ as the telescoping sum
\[
v_l^n = \sum_{k=1}^l (v_k^n - v_{k-1}^n).
\]
Consequently, we obtain
\begin{equation} \label{eq:sum_differential}
\begin{aligned} 
\left|v_l^n - \frac1n\sum_{k=1}^l \dd(\Exp_{x_0})_{v_{k-1}^n}^{-1}X_k^n\right| 
&\leq 
\sum_{k=1}^l |v_k^n - v_{k-1}^n -  \dd(\Exp_{x_0})_{v_{k-1}^n}^{-1}X_k^n| 
\\
&\leq
C\frac l{n^2}, 
\end{aligned}
\end{equation}
where the last line follows from the estimate in \eqref{eq:induction_step}. 

For the other term, observe that by Corollary \ref{corollary:inverse_exp_parallel}, there exists a constant $C > 0$ only depending on the compact set $\overline{B(0,r)}$ and $r$, such that
\begin{equation}\label{eq:replacing_diff_par}
|\dd(\Exp_{x_0})_{v_{k-1}^n}^{-1}X_k^n - \tau_{x_0(\frac1n*\Ss)_{k-1}}^{-1}X_k^n| \leq C|v_{k-1}^n|^2
\end{equation}
But then we find
\begin{align*}
\frac1n\sum_{k=1}^l|\dd(\Exp_{x_0})_{v_{k-1}^n}^{-1}X_k^n - \tau_{x_0(\frac1n*\Ss)_{k-1}}^{-1}X_k^n| 
&\leq 
C\frac1n\sum_{k=1}^l |v_{k-1}^n|^2
\\
&\leq
Cr^2\frac{l^3}{n^3},
\end{align*}
where in the last line we used that $|v_{k-1}^n| \leq r\frac{k-1}{n} \leq r\frac ln$ for any $1 \leq k \leq l$.
\end{proof}

\begin{remark}
The estimate in Proposition \ref{prop:comparison_independent_sum} is one of the most important ingredients of the proof of Theorem \ref{theorem:Cramer_sect_proof}. Indeed, it allows us in some sense to connect large deviations for $\{(\frac1n*\Ss)_n\}_{n\geq0}$ in $M$ to large deviations for the sums $\{\frac1n\sum_{k=1}^n \tau_{x_0(\frac1n*\Ss)_{k-1}}^{-1}X_k^n\}_{n\geq0}$ in the tangent space $T_{x_0}M$. Consequently, by making appropriate assumptions on the sequence $\{\frac1n\sum_{k=1}^n \tau_{x_0(\frac1n*\Ss)_{k-1}}^{-1}X_k^n\}_{n\geq0}$, for example in the spirit of the G\"artner-Ellis theorem (see e.g. \cite{Hol00,DZ98}), we can obtain more general results than Cram\'er's theorem for geodesic random walks in a similar way.
\end{remark}

One might hope to combine Propositions \ref{prop:LDP_Exp} and \ref{prop:comparison_independent_sum} to prove that $\{v_n^n\}_{n\geq0}$ satisfies in $T_{x_0}M$ the large deviation principle. Unfortunately, the upper bound found in Proposition \ref{prop:comparison_independent_sum} gives an unwanted contribution on the exponential scale. Indeed, taking $l = n$, we find that the upper bound in \eqref{eq:sum_parallel_transport} is $\Oo(1)$, which results in the fact that we get stuck with a constant as explained in Step 1 of Section \ref{section:sketch_upper_bound}. In an attempt to reduce this term in the upper bound, we cut up the random walk in finitely many pieces and analyse the pieces separately.

To this end, recall that
\[
d\left(\left(\frac1n*\Ss\right)_k,x_0\right) \leq \frac1n\sum_{l=1}^k |X_k^n| \leq \frac knr. 
\]
Now observe that $\iota(\overline{B(x_0,r)}) > 0$, because $\overline{B(x_0,r)}$ is compact (see \eqref{eq:injectivity_radius_set} for the definition of the injectivity radius of a set). Consequently, if $k \leq \frac{n\iota(\overline{B(x_0,r)})}{2r}$, then
\begin{equation}\label{eq:injectivity_radius}
d\left(\left(\frac1n*\Ss\right)_k,x_0\right) \leq \frac{\iota(\overline{B(x_0,r)})}{2} < \iota(\overline{B(x_0,r)}).
\end{equation}

Now let $m \in \NN$ such that $m \geq \frac{2r}{\iota(\overline{B(x_0,r)})}$. For $0 \leq l \leq m-1$ we define $n_l = l\lfloor m^{-1}n\rfloor$ and $n_m = n$. By \eqref{eq:injectivity_radius}, for every $0 \leq l \leq m-1$ and $1 \leq k \leq n_{l+1}-n_l$ we can uniquely define
\begin{equation}\label{eq:pullback_vectors}
\tilde v_k^{n,m,l} \in \Exp_{(\frac1n*\Ss)_{n_l}}^{-1}\left(\left(\frac1n*\Ss\right)_{n_l+k}\right) \subset T_{(\frac1n*\Ss)_{n_l}}M
\end{equation}
with $|\tilde v_k^{n,m,l}| < \iota((\frac1n*\Ss)_{n_l})$, because $n_{l+1} - n_l \leq nm^{-1} \leq \frac{n\iota(\overline{B(x_0,r)})}{2r}$. Finally, we set
\[
v_k^{n,m,l} = \tau_{x_0(\frac1n*\Ss)_{n_l}}^{-1}\tilde v_k^{n,m,l} \in T_{x_0}M,
\]
where parallel transport $\tau_{x_0(\frac1n*\Ss)_{n_l}}^{-1}$ is taken along the piecewise geodesic path through the points $(\frac1n*\Ss)_{n_1}, \ldots, (\frac1n*\Ss)_{n_{l-1}}$ .\\

Alongside this division of the random walk into pieces, we define a map $\Psi_m: (T_{x_0}M)^m \to M$ to identify the tuple $(v_{\lfloor m^{-1}n\rfloor}^{n,m,1},\ldots,v_{\lfloor m^{-1}n\rfloor}^{n,m,m})$ with $(\frac1n*\Ss)_n$, just like we used the Riemannian exponential map to identify $v_n^n$ and $(\frac1n*\Ss)_n$ before. Essentially, $\Psi_m$ is an $m$ time recursive application of the Riemannian exponential map. 

More precisely, let  $(v_1,\ldots,v_m) \in (T_{x_0}M)^m$ be given and define $x_1 = \Exp_{x_0}(v_1)$. Now, suppose $x_1,\ldots,x_i$ are given. Denote by $\tau_{x_0x_i}$ parallel transport along the constructed piecewise geodesic path via $x_1,\ldots,x_{i-1}$. Then we define $\tilde v_{i+1} = \tau_{x_0x_i}v_{i+1}$ and set $x_{i+1} = \Exp_{x_i}(\tilde v_{i+1})$. Finally, we define $\Psi_m(v_1,\ldots,v_m) = x_m$. In particular, we have for every $x \in M$ and $v \in \Exp_{x_0}^{-1}x$ that $(\frac1mv,\ldots,\frac1mv) \in \Psi_m^{-1}x$. To see this, observe that the path that $\Psi_m$ constructs is precisely the geodesic $\gamma_v(t) = \Exp_{x_0}(tv)$, because the speed of a geodesic is parallel along the geodesic. Furthermore, the map $\Psi_m$ is continuous as a composition of continuous maps.

\begin{remark}
Strictly speaking, if we divide the random walk into $m$ pieces as above, for the last piece we can only guarantee that it has at most $\lfloor m^{-1}n\rfloor + m$ increments, since $n$ need not be divisible by $m$. Additionally, this implies that $\Psi_m(v_{\lfloor m^{-1}n\rfloor}^{n,m,1},\ldots,v_{\lfloor m^{-1}n\rfloor}^{n,m,m})$ is only equal to $\left(\frac1n*\Ss\right)_n$ when $n$ is divisible by $m$. However, for every $m \in \NN$ it holds that
$$
d\left(\Psi_m(v_{\lfloor m^{-1}n\rfloor}^{n,m,1},\ldots,v_{\lfloor m^{-1}n\rfloor}^{n,m,m}),\left(\frac1n*\Ss\right)_n\right) = \Oo\left(\frac1n\right).
$$
Since in the proofs to follow we always first let $n$ tend to infinity before $m$, this has no influence on the results and arguments. Therefore, to avoid unnecessarily complicated notation and reasoning, we proceed with the above. 
\end{remark}

\subsection{Upper bound of the large deviation principle for $\{(\frac1n*\Ss)_n\}_{n\geq0}$}\label{section:upper_bound}
In this section we prove the large deviation upper bound for $\{(\frac1n*\Ss)_n\}_{n\geq0}$. Before we can do this, we first need some preliminary results.

\begin{proposition}[Upper bound for $\EE(e^{n\inp{\lambda}{v_n^n}})$]\label{prop:upper_bound_mgf}
Let the assumptions of Theorem \ref{theorem:Cramer_sect_proof} be satisfied. Additionally, let $r$ be the uniform bound of the increments and assume that $r < \iota(x_0)$. Then there exists a constanct $C > 0$ such that for all $n$ and all $1 \leq l \leq n$
\[
\EE(e^{n\inp{\lambda}{v_l^n}}) \leq e^{ln^{-1}|\lambda|C}e^{|\lambda| C r^2 l^3n^{-2}} M_{x_0}(\lambda)^l
\]
for all $\lambda \in T_{x_0}M$. Here, $M_{x_0}(\lambda) = \int_{T_{x_0}M} e^{\inp{\lambda}{v}}\mu_{x_0}(\dd v)$.
\end{proposition}
\begin{proof}
By Proposition \ref{prop:comparison_independent_sum} and the Cauchy-Schwarz inequality, there exists a constant $C > 0$ such that
\begin{align*}
\inp{\lambda}{v_l^n} - \frac1n\sum_{k=1}^l\inp{\lambda}{\tau_{x_0(\frac1n*\Ss)_{k-1}}^{-1}X_k^n}
&\leq
|\lambda|\left|v_l^n - \frac1n\sum_{k=1}^l \tau_{x_0(\frac1n*\Ss)_{k-1}}^{-1}X_k^n\right|
\\
&\leq
C|\lambda|\frac l{n^2} + C|\lambda|r^2\frac{l^3}{n^3}.
\end{align*}
But then we can estimate
\begin{align*}
\EE\left(e^{n\inp{\lambda}{v_l^n}}\right)
&=
\EE\left(e^{\sum_{k=1}^l {\inp{\lambda}{\tau_{x_0(\frac1n*\Ss)_{k-1}}^{-1}X_k^n}}}e^{n\inp{\lambda}{v_l^n} - \sum_{k=1}^l {\inp{\lambda}{\tau_{x_0(\frac1n*\Ss)_{k-1}}^{-1}X_k^n}}}\right) 
\\
&\leq
e^{C|\lambda|ln^{-1}}e^{C|\lambda|r^2l^3n^{-2}}\EE\left(e^{\sum_{k=1}^l {\inp{\lambda}{\tau_{x_0(\frac1n*\Ss)_{k-1}}^{-1}X_k^n}}}\right).
\end{align*}

As shown in the proof of Proposition \ref{prop:LDP_Exp}, for every $1 \leq k \leq n$ we have that $\tau_{x_0(\frac1n*\Ss)_{k-1}}^{-1}X_k^n$ is distributed accodring to $\mu_{x_0}$ and is independent of $\tau_{x_0(\frac1n*\Ss)_{l-1}}^{-1}X_l^n$ for any $l \neq k$. Consequentely, we find that
\begin{align*}
\EE\left(e^{\sum_{k=1}^l {\inp{\lambda}{\tau_{x_0(\frac1n*\Ss)_{k-1}}^{-1}X_k^n}}}\right)
&=
\prod_{k=1}^l \EE\left(e^{\inp{\lambda}{\tau_{x_0(\frac1n*\Ss)_{k-1}}^{-1}X_k^n}}\right)
\\
&=
M_{x_0}(\lambda)^l,
\end{align*}
where the last step follows from Proposition \ref{prop:logmgf}.

\end{proof}

Using Proposition \ref{prop:upper_bound_mgf}, we obtain the following inequality, which is key in deriving the large deviations upper bound for $\{(\frac1n*\Ss)_n\}_{n\geq0}$.   

\begin{proposition}\label{prop:upper_bound_LDP_tangent}
Let the assumptions of Theorem \ref{theorem:Cramer_sect_proof} be satisfied. Denote by $r$ the uniform bound on the increments of the geodesic random walk. Then for any $m \in \NN$ such that $m \geq \frac{2r}{\iota(\overline{B(x_0,r)})}$  and any closed $F \subset (T_{x_0}M)^m$ we have
\begin{align*}
&\limsup_{n\to\infty} \frac1n\log\PP\left(\left(v_{\lfloor m^{-1}n \rfloor}^{n,m,1},\ldots,v_{\lfloor m^{-1}n \rfloor}^{n,m,m}\right) \in F\right)
\\
&\leq 
-\inf_{(v_1,\ldots,v_m) \in F} \sup_{(\lambda_1,\ldots,\lambda_m) \in (T_{x_0}M)^m} \frac1m\sum_{i=1}^m \left\{\inp{\lambda_i}{mv_i} - \Lambda_{x_0}(\lambda_i) - m^{-2}C|\lambda_i|r^2\right\}.
\end{align*}
Here, $C$ is a constant depending on the curvature of the compact set $\overline{B(0,r)}$ and the bound $r$.
\end{proposition}
\begin{proof}
We first prove the upper bound for compact sets, so let $\Gamma \subset (T_{x_0}M)^m$ be compact. Following the proof of Cram\'er's theorem (see e.g. \cite{DZ98,Hol00}) for the vector space $(T_{x_0}M)^m$, we have
\begin{align*}
&\limsup_{n\to\infty} \frac1n\log\PP\left(\left(v_{\lfloor m^{-1}n \rfloor}^{n,m,1},\ldots,v_{\lfloor m^{-1}n \rfloor}^{n,m,m}\right) \in \Gamma\right)
\\
&\leq 
-\inf_{(v_1,\ldots,v_m) \in \Gamma} \sup_{(\lambda_1,\ldots,\lambda_m) \in (T_{x_0}M)^m} \left\{\sum_{i=1}^m \inp{\lambda_i}{v_i} - \limsup_{n\to\infty} \frac1n\log\EE\left(e^{n\sum_{i=1}^m \inp{\lambda_i}{v_{\lfloor m^{-1}n\rfloor}^{n,m,i}}}\right)\right\}.
\end{align*}

Recall that for $0 \leq i \leq m-1$ we write $n_i = i\lfloor m^{-1}n\rfloor$ and $n_m = n$. By Proposition \ref{prop:comparison_independent_sum} (which we may apply, because $m$ is chosen large enough) there exists a $C > 0$ such that for any $1 \leq i \leq m$ we have
\begin{align*}
\left|\tilde v_{\lfloor m^{-1}n \rfloor}^{n,m,i} - \frac1n\sum_{k=n_{i-1} + 1}^{n_i} \tau_{(\frac1n*\Ss)_{n_{i-1}}(\frac1n*\Ss)_k}^{-1}X_k^n\right| 
&\leq
C\frac {\lfloor m^{-1}n \rfloor}{n^2} + Cr^2\frac{\lfloor m^{-1}n \rfloor^3}{n^3}
\\
&\leq
C\frac{1}{nm} + Cr^2\frac1{m^3}.
\end{align*}
But then we also have that
\begin{equation}\label{eq:comparison_parallel_transport}
\left|v_{\lfloor m^{-1}n \rfloor}^{n,m,i} - \frac1n\tau_{x_0(\frac1n*\Ss)_{n_{i-1}}}^{-1}\sum_{k=n_{i-1} + 1}^{n_i} \tau_{(\frac1n*\Ss)_{n_{i-1}}(\frac1n*\Ss)_k}^{-1}X_k^n\right|  \leq C\frac{1}{nm} + Cr^2\frac1{m^3},
\end{equation}
because parallel transport is an isometry.

Now define
\[
Y_i^n = \tau_{x_0(\frac1n*\Ss)_{n_{i-1}}}^{-1}\sum_{k=n_{i-1} + 1}^{n_i} \tau_{(\frac1n*\Ss)_{n_{i-1}}(\frac1n*\Ss)_k}^{-1}X_k^n \in T_{x_0}M.
\]
Using \eqref{eq:comparison_parallel_transport}, it follows from the Cauchy-Schwarz inequality and the triangle inequality that
\[
\left|\sum_{i=1}^m \inp{\lambda_i}{v_{\lfloor m^{-1}n\rfloor}^{n,m,i}} - \frac1n\sum_{i=1}^m \inp{\lambda_i}{Y_i^n}\right| \leq C\left(\frac1{nm} + r^2\frac{1}{m^3}\right)\sum_{i=1}^m |\lambda_i|.
\]
Consequently, we find that
\[
\EE\left(e^{n\sum_{i=1}^m \inp{\lambda_i}{v_{\lfloor m^{-1}n\rfloor}^{n,m,i}}}\right) \leq e^{Cm^{-1}\sum_{i=1}^m|\lambda_i|}e^{Cr^2m^{-3}n\sum_{i=1}^m |\lambda_i|}\EE\left(e^{\sum_{i=1}^m \inp{\lambda_i}{Y_i^n}}\right).
\]

Now note that, like in the proof of Proposition \ref{prop:LDP_Exp}, we can show that for $i \neq j$ the random variables $Y_i^n$ and $Y_j^n$ are independent. Consequently, we have that
\[
\EE\left(e^{\sum_{i=1}^m \inp{\lambda_i}{Y_i^n}}\right) = \prod_{i=1}^m\EE\left(e^{\inp{\lambda_i}{Y_i^n}}\right).
\]
Moreover, again following the proof of Proposition \ref{prop:LDP_Exp}, one can show that
\[
\EE\left(e^{\inp{\lambda_i}{Y_i^n}}\right) = M_{x_0}(\lambda_i)^{\lfloor m^{-1}n\rfloor}.
\]

Combining everything, we find that
\begin{align*}
&\limsup_{n\to\infty} \frac1n\log\EE\left(e^{n\sum_{i=1}^m \inp{\lambda_i}{v_{\lfloor m^{-1}n\rfloor}^{n,m,i}}}\right)
\\
&\leq
\limsup_{n\to\infty} \left\{\frac{C}{mn}\sum_{i=1}^m|\lambda_i| + \frac{Cr^2}{m^3}\sum_{i=1}^m|\lambda_i| + \frac{\lfloor m^{-1}n\rfloor}n\sum_{i=1}^m \Lambda_{x_0}(\lambda_i)\right\} 
\\
&=
\frac{Cr^2}{m^3}\sum_{i=1}^m|\lambda_i| + \frac1m\sum_{i=1}^m \Lambda_{x_0}(\lambda_i).
\end{align*}

Putting everything together, we obtain
\begin{align*}
&\limsup_{n\to\infty} \frac1n\log\PP\left(\left(v_{\lfloor m^{-1}n \rfloor}^{n,m,1},\ldots,v_{\lfloor m^{-1}n \rfloor}^{n,m,m}\right) \in \Gamma\right)
\\
&\leq 
-\inf_{(v_1,\ldots,v_m) \in \Gamma} \sup_{(\lambda_1,\ldots,\lambda_m) \in (T_{x_0}M)^m} \sum_{i=1}^m\left\{ \inp{\lambda_i}{v_i} - m^{-1}\Lambda_{x_0}(\lambda_i) - m^{-3}Cr^2|\lambda_i|\right\}.
\end{align*}
This concludes the proof of the upper bound for compact sets.

To extend the upper bound to all closed sets, one should simply notice that 
\[
\left(v_{\lfloor m^{-1}n \rfloor}^{n,m,1},\ldots,v_{\lfloor m^{-1}n \rfloor}^{n,m,m}\right) \in \overline{B(0,r)}^m
\]
almost surely, where $r$ is the uniform bound of the increments. Since $M$ is complete, $\overline{B(0,r)}^m$ is compact, so that the sequence is exponentially tight. 
\end{proof}

It now remains to transfer the upper bound in Proposition \ref{prop:upper_bound_LDP_tangent}  for the process in $(T_{x_0}M)^m$ related to $\left\{\left(\frac1n*\Ss\right)_n\right\}_{n\geq0}$ to the upper bound of the large deviation principle for $\left\{\left(\frac1n*\Ss\right)_n\right\}_{n\geq0}$. With all preperations done, the only thing that remains to be shown, is that the upper bound has the desired form.

\begin{proposition}\label{prop:upper_bound_LDP}
Let the assumptions of Theorem \ref{theorem:Cramer_sect_proof} be satisfied. Then for any $F \subset M$ closed we have
\[
\limsup_{n\to\infty} \frac1n\log\PP\left(\left(\frac1n*\Ss\right)_n \in F\right) \leq -\inf_{x\in F} I_M(x),
\]
where 
\[
I_M(x) = \inf\{\Lambda_{x_0}^*(v)| v \in \Exp_{x_0}^{-1}x\}.
\]
\end{proposition}
\begin{proof}
Let $F \subset M$ be closed and pick $m \in \NN$ such that $m \geq \frac{2r}{\iota(\overline{B(x_0,r)})}$, where $r$ denotes the uniform bound of the increments. Let $\Psi_m :(T_{x_0}M)^m \to M$ be the recursive application of the Riemannian exponential map defined just above Section \ref{section:upper_bound}. Because $\Psi_m$ is continuous, we have that $\Psi_m^{-1}F \subset (T_{x_0}M)^m$ is closed. Hence, by Proposition \ref{prop:upper_bound_LDP_tangent} we find that
\begin{align*}
&\limsup_{n\to\infty} \frac1n\log\PP\left(\left(\frac1n*\Ss\right)_n \in F\right)
\\
&\leq
\limsup_{n\to\infty} \frac1n\log\PP\left(\left(v_{\lfloor m^{-1}n \rfloor}^{n,m,1},\ldots,v_{\lfloor m^{-1}n \rfloor}^{n,m,m}\right) \in \Psi_m^{-1}F\right)
\\
&\leq 
-\inf_{(v_1,\ldots,v_m) \in \Psi_m^{-1}F} \sup_{(\lambda_1,\ldots,\lambda_m) \in (T_{x_0}M)^m} \frac1m\sum_{i=1}^m \left\{\inp{\lambda_i}{mv_i} - \Lambda_{x_0}(\lambda_i) - m^{-2}Cr^2|\lambda_i|\right\}.
\end{align*}
Now observe that for every $\lambda \in T_{x_0}M$ we have $|\lambda| \leq |\lambda|^2 + 1$. Plugging this into the above estimate, keeping in mind the minus sign in front, we find that
\begin{align*}
&\limsup_{n\to\infty} \frac1n\log\PP\left(\left(\frac1n*\Ss\right)_n \in F\right)
\\
&\leq 
\frac{Cr^2}{m^2} - \inf_{(v_1,\ldots,v_m) \in \Psi_m^{-1}F} \sup_{(\lambda_1,\ldots,\lambda_m) \in (T_{x_0}M)^m} \frac1m\sum_{i=1}^m \left\{\inp{\lambda_i}{mv_i} - \Lambda_{x_0}(\lambda_i) - m^{-2}Cr^2|\lambda_i|^2\right\}.
\end{align*}

We now focus on the infimum in the above expression. The necessity of replacing $|\lambda|$ with $|\lambda|^2$, and making the upper bound slightly worse, will become clear when we try to calculate this infimum further. 

 First, consider the map $\Lambda_m: T_{x_0}M \to \RR$ defined by
\[
\Lambda_m(\lambda) = \Lambda_{x_0}(\lambda) + \frac1{m^2}Cr^2|\lambda|^2.
\]
and denote by $\Lambda_m^*$ its Legendre transform. Then
\[
\sup_{(\lambda_1,\ldots,\lambda_m) \in (T_{x_0}M)^m} \frac1m\sum_{i=1}^m \left\{\inp{\lambda_i}{mv_i} - \Lambda_{x_0}(\lambda_i) - m^{-2}Cr^2|\lambda_i|\right\} = \frac1m\sum_{i=1}^m \Lambda_m^*(mv_i).
\]
The latter may be interpreted as 
\[
\int_0^1 \Lambda_m^*(\dot\gamma_m(t))\ud t,
\]
where $\gamma_m$ is piecewise geodesic on intervals of the form $[\frac{(i-1)}m,\frac im]$ with speed $m\tilde v_i$, where $\tilde v_i = \tau_{x_0\gamma_m(\frac{(i-1)}m)}v_i$. 

Now note that since $\Lambda_{x_0}$ is differentiable and convex, we find that $\Lambda_m$ is differentiable and strictly convex. Furthermore, we have for every $u \in T_{x_0}M$ that
\[
\Lambda_m^*(u) = \sup_{\lambda \in T_{x_0}M} \left\{\inp{\lambda}{u} - \Lambda_{x_0}(\lambda) - \frac1{m^2}Cr^2|\lambda|^2\right\} \leq \sup_{\lambda \in T_{x_0}M} \left\{\inp{\lambda}{u} - \frac1{m^2} Cr^2|\lambda|^2\right\} < \infty.
\]
Here we used that $\Lambda_{x_0}$ is non-negative, because the expectation of $\mu_{x_0}$ is 0. Consequently, we find that $\Lambda_m^*$ is everywhere finite. Note that this does not contradict the fact that the rate function might be infinite, since $\Lambda_m^*$ merely provides a lower bound of the rate function.  Because $\Lambda_m^*$ is everywhere finite, it follows from Lemma \ref{lemma:convex_analysis} that $\Lambda_m^*$ is strictly convex and differentiable. 

The above shows that we can apply \cite[Proposition 8.3]{KRV18}, giving us that minimizing trajectories for the functional
\[
\int_0^1 \Lambda_m^*(\dot\gamma(t))\ud t
\] 
are geodesics. Because for every $x \in F$ and every $v \in \Exp_{x_0}^{-1}x$ we have that $(\frac1mv,\ldots,\frac1mv) \in \Psi_m^{-1}F$, we find that
\begin{align*}
&-\inf_{(v_1,\ldots,v_m) \in \Psi_m^{-1}F} \sup_{(\lambda_1,\ldots,\lambda_m) \in (T_{x_0}M)^m} \frac1m\sum_{i=1}^m \left\{\inp{\lambda_i}{mv_i} - \Lambda_{x_0}(\lambda_i)\} - m^{-2}Cr^2|\lambda_i|^2\right\}
\\
&=
-\inf_{v\in \Exp_{x_0}^{-1}F} \sup_{\lambda \in T_{x_0}M} \left\{\inp{\lambda}{v} - \Lambda_{x_0}(\lambda) - m^{-2}Cr^2|\lambda|^2\right\}.
\end{align*}
Now note that 
\begin{align}
&\lim_{m\to\infty} \sup_{\lambda \in T_{x_0}M} \left\{\inp{\lambda}{v} - \Lambda_{x_0}(\lambda) - m^{-2}Cr^2|\lambda|^2\right\} \nonumber
\\
&=
\sup_{\lambda \in T_{x_0}M} \lim_{m\to\infty}\left\{\inp{\lambda}{v} - \Lambda_{x_0}(\lambda) - m^{-2}Cr^2|\lambda|^2\right\} \label{eq:lim_logmgf}
\\
&=
\sup_{\lambda \in T_{x_0}M} \left\{\inp{\lambda}{v} - \Lambda_{x_0}(\lambda)\right\}, \nonumber
\end{align}
because $\Lambda_m(\lambda) = \inp{\lambda}{v} - \Lambda_{x_0}(\lambda) - m^{-2}Cr^2|\lambda|^2$ is increasing in $m$ for every $\lambda \in T_{x_0}M$. Furthermore, we have
\[
\inp{\lambda}{v} - \Lambda_{x_0}(\lambda) - m^{-2}Cr^2|\lambda| \geq \inp{\lambda}{v} - r|\lambda| - m^{-2}Cr^2|\lambda|^2,
\]
because the support of $\mu_{x_0}$ is contained in $\overline{B(0,r)}$. Furthermore, one may compute that if $|v| > r$, then
\begin{equation}\label{eq:growth_rate}
\sup_{\lambda \in T_{x_0}M} \left\{ \inp{\lambda}{v} - r|\lambda| - m^{-2}Cr^2|\lambda|^2\right\} = \frac{m^2}{4Cr^2}(|v| - r)^2.
\end{equation}

Now write 
\[
\Exp_{x_0}^{-1}F = \left(\Exp_{x_0}^{-1}F \cap \overline{B(0,2r)}\right) \cup \left(\Exp_{x_0}^{-1}F \cap \overline{B(0,2r)}^C\right). 
\]
Note that by \eqref{eq:growth_rate}, we find that
\begin{align*}
&\lim_{m\to\infty} \inf_{v\in\Exp_{x_0}^{-1}F \cap \overline{B(0,2r)}^C} \sup_{\lambda \in T_{x_0}M} \left\{\inp{\lambda}{v} - \Lambda_{x_0}(\lambda) - m^{-2}Cr^2|\lambda|^2\right\}
\\
&\geq
\lim_{m\to\infty} \inf_{v\in\Exp_{x_0}^{-1}F \cap \overline{B(0,2r)}^C} \sup_{\lambda \in T_{x_0}M} \left\{\inp{\lambda}{v} - r|\lambda| - m^{-2}Cr^2|\lambda|^2\right\}
\\
&\geq
\lim_{m\to\infty} \frac{m^2}{4Cr^2}r^2
\\
&=
\infty,
\end{align*}
where we used in the one but last line that $|v| \geq 2r$. Also, because $|v| \geq 2r \geq r$, we have
\[
\sup_{\lambda \in T_{x_0}M} \left\{\inp{\lambda}{v} - \Lambda_{x_0}(\lambda)\right\} = \infty, 
\]
so that
\begin{align*}
&\lim_{m\to\infty} \inf_{v\in\Exp_{x_0}^{-1}F \cap \overline{B(0,2r)}^C} \sup_{\lambda \in T_{x_0}M} \left\{\inp{\lambda}{v} - \Lambda_{x_0}(\lambda) - m^{-2}Cr^2|\lambda|^2\right\}
\\
&=
 \inf_{v\in\Exp_{x_0}^{-1}F \cap \overline{B(0,2r)}^C} \sup_{\lambda \in T_{x_0}M} \left\{\inp{\lambda}{v} - \Lambda_{x_0}(\lambda)\right\}.
\end{align*}

For the other part, because $\Exp_{x_0}^{-1}F \cap \overline{B(0,2r)}$ is compact, it follows from \eqref{eq:lim_logmgf} that
\begin{align*}
&\lim_{m\to\infty} \inf_{v\in \Exp_{x_0}^{-1}F\cap \overline{B(0,2r)}} \sup_{\lambda \in T_{x_0}M} \left\{\inp{\lambda}{v} - \Lambda_{x_0}(\lambda) - m^{-2}Cr^2|\lambda|^2\right\}
\\
&= 
\inf_{v\in \Exp_{x_0}^{-1}F\cap \overline{B(0,2r)}} \sup_{\lambda \in T_{x_0}M} \left\{\inp{\lambda}{v} - \Lambda_{x_0}(\lambda)\right\}.
\end{align*}

Collecting everything, we find that
\begin{align*}
&\limsup_{n\to\infty} \frac1n\log\PP\left(\left(\frac1n*\Ss\right)_n \in F\right)
\\
&\leq
\lim_{m\to\infty} \left( \frac{Cr^2}{m^2} -  \inf_{v\in \Exp_{x_0}^{-1}F} \sup_{\lambda \in T_{x_0}M} \left\{\inp{\lambda}{v} - \Lambda_{x_0}(\lambda) - m^{-2}Cr^2|\lambda|^2\right\}\right)
\\
&=
-\lim_{m\to\infty} \inf_{v\in \Exp_{x_0}^{-1}F\cap \overline{B(0,2r)}} \sup_{\lambda \in T_{x_0}M} \left\{\inp{\lambda}{v} - \Lambda_{x_0}(\lambda) - m^{-2}Cr^2|\lambda|^2\right\}
\\
&=
-\inf_{v\in \Exp_{x_0}^{-1}F} \sup_{\lambda \in T_{x_0}M} \left\{\inp{\lambda}{v} - \Lambda_{x_0}(\lambda)\right\}
\\
&=
-\inf_{x\in F} I_M(x),
\end{align*}
which concludes the proof.
\end{proof}

\subsection{Lower bound of the large deviation principle for $\{(\frac1n*\Ss)_n\}_{n\geq0}$}\label{section:lower_bound}
In this section we prove the large deviation lower bound for $\{(\frac1n*\Ss)_n\}_{n\geq0}$. In order to do this, we need a refinement of Proposition \ref{prop:LDP_Exp}, which may be proven in a similar way.

\begin{proposition}\label{prop:LDP_Exp_m}
Let the assumptions of Theorem \ref{theorem:Cramer_sect_proof} be satisfied. Let $m \in \NN$ and set $Z_n^m = \frac1n\sum_{k=1}^{{\lfloor m^{-1}n\rfloor}} \tau_{x_0(\frac1n*\Ss)_{k-1}}^{-1}X_k^n$. Finally, define $\Lambda_{x_0}(\lambda) = \log\EE(e^{\inp{\lambda}{X_1}})$. Then $\{Z_n^m\}_{n\geq1}$ satisfies in $T_{x_0}M$ the large deviation principle with good rate function
\[
I_m(v) = \frac1m\Lambda_{x_0}^*(mv), 
\]
where $\Lambda_{x_0}^*(v) = \sup_{\lambda \in T_{x_0}M} \{\inp{\lambda}{v} - \Lambda_{x_0}(\lambda)\}$.
\end{proposition}

We are now able to prove the large deviations lower bound for $\{(\frac1n*\Ss)_n\}_{n\geq0}$.

\begin{proposition}\label{prop:lower_bound_LDP}
Let the assumptions of Theorem \ref{theorem:Cramer_sect_proof} be satisfied. 
Then for any $G \subset M$ open, 
\[
\liminf_{n\to\infty} \frac1n\log \PP\left(\left(\frac1n*\Ss\right)_n \in G\right) \geq -\inf_{x \in G} I_M(x),
\]
where $I_M$ is as in \eqref{eq:rate_cramer_sect_proof}.
\end{proposition}
\begin{proof}
It suffices to show that 
\[
\liminf_{n\to\infty} \frac1n\log\PP\left(\left(\frac1n*\Ss\right)_n \in G\right) \geq -I_M(x)
\]
for every $x \in G$.

So fix $x \in G$ and pick $v \in \Exp_{x_0}^{-1}x$. Because $G$ is open, there exists an $\epsilon > 0$ such that $B(x,\epsilon) \subset G$. Let $m \in \NN$ such that $m \geq \frac{2r}{\iota(\overline{B(x_0,r)})}$, where $r$ is the uniform bound on the increments of the geodesic random walk. 

We again need to identify the geodesic random walk with a tuple in $(T_{x_0}M)^m$. However, this time the parallel transport back to $T_{x_0}M$ is carried out by first transporting to a well-chosen point on the geodesic $\gamma_v(t) = \Exp_{x_0}(tv)$ and then to $x_0$ along this geodesic. 

More precisely, we define a map $\Psi_{m,x,v}:(T_{x_0}M)^m \to M$ that allows us to identify the random variable $(\frac1n*\Ss)_n \in M$ with a vector of random variables in $(T_{x_0}M)^m$. To this end, define for $0 \leq i \leq m$ the points $y_i = \Exp_{x_0}(\frac im v)$. For $1 \leq i \leq m$ we define $\tau_{x_0y_i}$ as parallel transport along the geodesic $\Exp_{x_0}(tv)$. Furthermore, for every $z \in M$ and every $0 \leq i \leq m-1$, we choose a geodesic $\gamma_{y_ix}$ of minimum length and denote by $\tau_{y_ix}$ parallel transport along this geodesic.  We now define $\Psi_{m,x,v}(v_1,\ldots,v_m)$ as follows. Define $x_1 = \Exp_{x_0}(\frac1mv_1)$ and if $x_i$ is defined, we set $\tilde v_{i+1} = \tau_{y_ix_i}\tau_{x_0y_i}v_i$ and define $x_{i+1} = \Exp_{x_i}(\frac1m\tilde v_{i+1})$. Finally, we set $\Psi_{m,x,v}(v_1,\ldots,v_m) = x_m$. \\

Now note that by the triangle inequality, we have
\[
d(x_i,x_0) \leq \frac1m\sum_{j=1}^i |v_j| \leq \frac1m\sum_{j=1}^m |v_j|
\]
for any $1 \leq i \leq m$. Consequently, if $(v_1,\ldots,v_m) \in B(v,1)^m$ we have $|v_j| \leq |v| + 1$, so that  
\[
d(x_i,x_0) \leq |v| + 1
\]
for any $1 \leq i \leq m$. Because also $d(x_0,y_i) \leq \frac im|v| \leq |v|$, we find that $x_i,y_i \in \overline{B(x_0,|v| + 1)}$ for all $1 \leq i \leq m$. 

Writing $\eta = |v| + 1$, we will show that there exists a constant $m_0 \in \NN$ such that for all $m \geq m_0$ we have
\begin{equation}\label{eq:implication_small_vectors}
(v_1,\ldots,v_m) \in B(v,\epsilon^2/(8\eta))^m \Rightarrow \Psi_{m,x,v}(v_1,\ldots,v_m) \in B(x,\epsilon),
\end{equation}
whenever $\epsilon > 0$ is small enough.

To this end, let $K \subset M$ be a compact set, such that all geodesics of minimal length between points $x,y \in \overline{B(x_0,\eta)}$ are contained in $K$. Because $K$ is compact, its injectivity radius $\iota(K)$ is strictly positive. 

Fix $0 < \delta < \iota(K)$. We first show that for $\epsilon$ small enough and $m$ large enough we have
\begin{equation}\label{eq:recusion_estimate}
d(x_i,y_i)^2 \leq \frac{i-1}{2m}\epsilon^2 + \frac i{m^2}C
\end{equation}
for $1 \leq i \leq m$. Here, $C>0$ is some constant only depending on $K$ and $\delta$. We proceed by induction.

First consider the case $i = 1$. By taking $m$ large enough, we can apply Proposition \ref{proposition:distance_geodesic_point} to obtain a constant $C > 0$ (depending only on $K$ and $\delta$) such that
\[
d(x_1,y_1)^2 = d\left(\Exp_{x_0}\left(\frac1mv_1\right),\Exp_{x_0}\left(\frac1mv\right)\right) \leq \frac1{m^2}C.
\]

Now suppose that $d(x_i,y_i)^2 \leq \frac{i-1}{2m}\epsilon^2  + \frac i{m^2}C$. Then in particular we have
\[
d(x_i,y_i)^2 \leq \frac{\epsilon^2}2 + \frac1m C,
\]  
which can be made smaller than $\frac\delta2$ by taking $\epsilon$ sufficiently small and $m$ sufficiently large. In that case, we may again apply Proposition \ref{proposition:distance_geodesic_point}, so that for the same constant $C > 0$ as above, we have
\begin{align*} 
d(x_{i+1},y_{i+1})^2
&=
d\left(\Exp_{x_i}\left(\frac1m\tau_{y_ix_i}\tau_{x_0y_i}v_{i+1}\right),\Exp_{y_i}\left(\frac1m\tau_{x_0y_i}v\right)\right)
\\
&\leq 
d(x_i,y_i)^2 + 2\frac1m\inp{\tau_{x_0y_i}v_{i+1} - \tau_{x_0y_i}v}{\Exp_{y_i}^{-1}x_i} + \frac1{m^2}C
\\
&\leq
\frac{i-1}{2m}\epsilon^2  + \frac i{m^2}C + 2\frac1m|\tau_{x_0y_i}v_{i+1} - \tau_{x_0y_i}v||\Exp_{y_i}^{-1}x_i| + \frac1{m^2}C
\\
&=
\frac{i-1}{2m}\epsilon^2  + \frac {i+1}{m^2}C + \frac2m|v_i - v|d(x_i,y_i). 
\end{align*}
Now, observe that $d(x_i,y_i) \leq 2\eta$ since $x_i,y_i \in \overline{B(x_0,\eta)}$. Using this, together with the induction hypothesis and the fact that $|v_i - v| \leq \frac{\epsilon^2}{8\eta}$, we find
\[
d(x_i,y_i)^2 
\leq
\frac{i-1}{2m}\epsilon^2  + \frac {i+1}{m^2}C + \frac1{2m}\epsilon^2  = \frac{2i}{2m}\epsilon^2 + \frac {i+1}{m^2}C,
\]
as desired. 

Now taking $i = m$ in \eqref{eq:recusion_estimate}, we obtain
\[
d(x_m,y_m)^2 \leq \frac{\epsilon^2}2 +  C\frac1m,
\] 
whenever $(v_1,\ldots,v_m) \in B(v,\epsilon^2/(8\eta))^m$. Consequently, if we take $m_0 > \frac{2C}{\epsilon^2}$, we obtain for $m > m_0$ that
\[
d\left(\Psi_{m,x,v}(v_1,\ldots,v_m),x\right)^2 = d(x_m,y_m)^2 < \frac{\epsilon^2}2 + \frac{\epsilon^2}2 = \epsilon^2
\]
as desired.\\

Fixing $m_0$ and $C$ as above, let $m \geq m_0$ be large enough so that we can define
\[
\tilde v_k^{n,m,l} \in \Exp_{(\frac1n*\Ss)_{n_l}}^{-1}\left(\left(\frac1n*\Ss\right)_{n_l+k}\right) \subset T_{(\frac1n*\Ss)_{n_l}}M
\]
like in \eqref{eq:pullback_vectors}. Different from before, we now define the vectors
\begin{equation}\label{eq:pullback_vectors_tangent}
v_k^{n,m,l} = \tau_{x_0y_{n_l}}^{-1}\tau_{y_{n_l}(\frac1n*\Ss)_{n_l}}^{-1}\tilde v_k^{n,m,l} \in T_{x_0}M,
\end{equation}
using the parallel transport procedure used in the definition of the map $\Psi_{m,x,v}$.

Consequently, by construction we obtain
\[
\Psi_{m,x,v}\left(v_{\lfloor m^{-1}n \rfloor}^{n,m,1},\ldots,v_{\lfloor m^{-1}n \rfloor}^{n,m,m}\right) = \left(\frac1n*\Ss\right)_n.
\]
Using this, together with the implication in \eqref{eq:implication_small_vectors}, we find 
\begin{align*}
\PP\left(\left(\frac1n*\Ss\right)_n \in G\right)
&\geq
\PP\left(\left(\frac1n*\Ss\right)_n \in B(x,\epsilon)\right)
\\
&\geq
\PP\left(\left(v_{\lfloor m^{-1}n \rfloor}^{n,m,1},\ldots,v_{\lfloor m^{-1}n \rfloor}^{n,m,m}\right) \in B(v,\epsilon^2/(8\eta))^m\right)
\end{align*}

Now define for $1 \leq i \leq m$ the random variables
\[
Y_i^n = \tau_{x_0y_{n_{i-1}}}^{-1}\tau_{y_{n_{i-1}}(\frac1n*\Ss)_{n_{i-1}}}^{-1}\sum_{k=n_{i-1} + 1}^{n_i} \tau_{(\frac1n*\Ss)_{n_{i-1}}(\frac1n*\Ss)_{k-1}}^{-1}X_k^n \in T_{x_0}M,
\]
where the parallel transport $\tau_{(\frac1n*\Ss)_{n_{i-1}}(\frac1n*\Ss)_{k-1}}^{-1}$ is carried out along the trajectory of the geodesic random walk. The sum is then transported from $T_{(\frac1n*\Ss)_{n_{i-1}}}M$ to $T_{x_0}M$ as in the definition of $v_k^{n,m,l}$ as in \eqref{eq:pullback_vectors_tangent}.

In the same way as we obtained \eqref{eq:comparison_parallel_transport} in the proof of Proposition \ref{prop:upper_bound_LDP_tangent}, we find that there exists a constant $\tilde C > 0$ such that
\[
\left|v_{\lfloor m^{-1}n \rfloor}^{n,m,1} - Y_i^n\right| \leq \tilde C\frac1{nm} + \tilde C r^2\frac{1}{m^3}.
\]
Consequently, we may take $m$ large enough such that almost surely we have
\[
\left|v_{\lfloor m^{-1}n \rfloor}^{n,m,1} - Y_i^n\right| < \frac{\epsilon^2}{16\eta}.
\]
But then we find that if $Y_i^n \in B(v,\epsilon^2/(16\eta))$, then $v_{\lfloor m^{-1}n \rfloor}^{n,m,1} \in B(v,\epsilon^2/(8\eta))$. This implies that
\begin{align*}
&\PP\left(\left(v_{\lfloor m^{-1}n \rfloor}^{n,m,1},\ldots,v_{\lfloor m^{-1}n \rfloor}^{n,m,m}\right) \in B(v,\epsilon^2/(8\eta))^m\right)
\\
&\geq 
\PP\left(\left(Y_1^n,\ldots,Y_m^n\right) \in B(v,\epsilon^2/(16\eta))^m\right).
\end{align*}

Now note that, like in the proof of Proposition \ref{prop:LDP_Exp}, we can show that the random variables $Y_i^n$ and $Y_j^n$ are independent and identically distributed for $i \neq j$, so that
\begin{align*}
\PP\left(\left(Y_1^n,\ldots,Y_m^n\right) \in B(v,\epsilon^2/(16\eta))^m\right) 
&= 
\prod_{i=1}^m \PP\left(Y_i^n \in B(v,\epsilon^2/(16\eta))\right)
\\
&=
\PP\left(Y_1^n \in B(v,\epsilon^2/(16\eta))\right)^m.
\end{align*}
Furthermore, by Proposition \ref{prop:LDP_Exp_m} we have
\[
\liminf_{n\to \infty} \frac1n\log\PP\left(Y_1^n \in B(v,\epsilon^2/(16\eta))\right) \geq -\frac1m\Lambda_{x_0}^*(v).
\]

Combining everything, we find that
\begin{align*}
\liminf_{n\to\infty}\frac1n\log\PP\left(\left(\frac1n*\Ss\right)_n \in G\right)
&\geq
m\liminf_{n\to\infty} \frac1n\log\PP\left(Y_1^n \in B(v,\epsilon^2/(16\eta))\right)
\\
&\geq
-\Lambda_{x_0}^*(v).
\end{align*}
Since this holds for all $v \in \Exp_{x_0}^{-1}x$, we find that
\[
 \liminf_{n\to\infty} \frac1n\log \PP\left(\left(\frac1n*\Ss\right)_n \in G\right) \geq -\inf_{v\in \Exp_{x_0}^{-1}x} \Lambda_{x_0}^*(v) = -I_M(x),
\]
which concludes the proof.
\end{proof}


\appendix

\section{Some convex analysis}

In this appendix we collect a result from convex analysis. Although this is probably well-known, we provide a proof for the reader's convenience.

\begin{lemma}\label{lemma:convex_analysis}
Let $V$ be a vector space, and let $F:V \to \RR$ be strictly convex and differentiable. Then its Legendre transform $F^*$ is strictly convex and differentiable on the interior of its domain $\Dd_{F^*}^{\circ}$.
\end{lemma}
\begin{proof}
The differentiability of $F^*$ follows from \cite{Roc70}[Theorem 26.3]. 

For the strict convexity, we first prove that for each $v \in \Dd_{F^*}^\circ$, there exists a $\lambda_v^* \in V$ such that
\[
F^*(v) = \inp{\lambda_v^*}{v} - F(\lambda_v^*).
\]
Indeed, suppose this is not the case. Because $F^*(v) < \infty$, we can find a sequence $\lambda_n$ such that
\[
F^*(v) = \lim_{n\to\infty} \inp{\lambda_n}{v} - F(\lambda_n).
\]
Because the map $\lambda \mapsto \inp{\lambda}{v} - F_x(\lambda)$ is continuous, the sequence $\lambda_n$ cannot contain a convergent subsequence, else the limit of this subsequence would serve as $\lambda_v^*$. Consequently, we must have that $\lim_{n\to\infty} |\lambda_n| = \infty$. 

But then there exists a $w \in V$ such that $\lim_{n\to\infty} \inp{\lambda_n}{w} = \infty$. To see this, suppose such a $w$ does not exist. Denoting by $e_1,\ldots,e_d$ a basis of $V$, we must have that $\inp{\lambda_n}{e_i}$ is a bounded sequence for all $i = 1,\ldots,d$. But then, by taking subsequences, we find $\inp{\lambda_n}{e_i}$ converges for all $i = 1,\ldots,d$, which contradicts the fact that $\lim_{n\to\infty} |\lambda_n| = \infty$. 

Now consider $v + \epsilon w \in V$ and let $\lambda_n$ be the sequence found above. We have that
\[
F^*(v+\epsilon w) \geq \lim_{n\to\infty} \inp{\lambda_n}{v + \epsilon w} - F(\lambda_n) = F^*(v) + \epsilon\lim_{n\to\infty} \inp{\lambda_n}{w} = \infty.
\]
We conclude that $v + \epsilon w \notin \Dd_{F^*}$ for any $\epsilon > 0$, which contradicts the assumption that $v \in \Dd_{F^*}^\circ$.\\

We are now ready to prove that $F^*$ is strictly convex on $\Dd_{F^*}^\circ$. To this end, fix $v,w \in \Dd_{F^*}^\circ$, $v \neq w$ and $t \in (0,1)$ and assume that
\begin{equation}\label{eq:equality_convex}
F^*(tv + (1-t)w) = tF^*(v) + (1-t)F^*(w).
\end{equation}
Now let $\lambda_t^*$ be such that
\[
F^*(tv + (1-t)w) = \inp{tv + (1-t)w}{\lambda_t^*} - F(\lambda_t^*).
\]
We find that
\[
tF^*(v) + (1-t)F^*(w) = t(\inp{\lambda_t^*}{v} - F(\lambda_t^*)) + (1-t)(\inp{\lambda_t^*}{w} - F(\lambda_t^*)).
\]
But then we find that
\[
F^*(v) = \inp{v}{\lambda_t^*} - F(\lambda_t^*)
\]
and
\[
F^*(w) = \inp{w}{\lambda_t^*} - F(\lambda_t^*).
\]
Now, because $F$ is everywhere differentiable, it must be that $\Nabla F(\lambda_t^*) = v$ and $\Nabla F(\lambda_t^*) = w$, which contradicts the assumption that $v \neq w$. We conclude that $F^*$ is strictly convex on $\Dd_{F^*}^\circ$.
\end{proof}




{\em Acknowledgment} -- The author thanks Frank Redig for helpful discussions, suggestions for certain arguments and the verification of various proofs. Thanks goes also to Jan van Neerven for suggestions on improving the clarity of certain definitions.

\end{document}